\documentclass[11pt,reqno]{amsart}

\usepackage[T1]{fontenc}
\usepackage{amsmath,amssymb,amsfonts,amsthm,mathtools}
\usepackage{mathrsfs}
\usepackage{enumitem}
\usepackage{microtype}
\usepackage{xcolor}
\usepackage[hidelinks,bookmarksdepth=3]{hyperref}
\usepackage[nameinlink,capitalize]{cleveref}
\usepackage[a4paper,total={164mm,237mm},left=23mm,top=30mm]{geometry}

\numberwithin{equation}{section}

\theoremstyle{plain}
\newtheorem{thm}{Theorem}[section]
\newtheorem{cor}[thm]{Corollary}
\newtheorem{lem}[thm]{Lemma}
\newtheorem{prop}[thm]{Proposition}

\theoremstyle{definition}
\newtheorem{defn}[thm]{Definition}

\theoremstyle{remark}
\newtheorem{rem}[thm]{Remark}
\crefname{thm}{Theorem}{Theorems}
\crefname{lem}{Lemma}{Lemmas}
\crefname{prop}{Proposition}{Propositions}
\crefname{cor}{Corollary}{Corollaries}
\crefname{defn}{Definition}{Definitions}
\crefname{rem}{Remark}{Remarks}

\newcommand{\F}{\mathcal F}
\newcommand{\G}{\mathcal G}
\newcommand{\W}{\mathcal W}
\newcommand{\C}{\mathcal C}
\newcommand{\U}{\mathcal U}

\newcommand{\Cover}{\mathrm{Cover}}
\newcommand{\SmallN}{\mathrm{Small}}
\newcommand{\Bin}{\mathrm{Bin}}
\newcommand{\E}{\mathbb E}

\newcommand{\Prob}{\mathbb P}

\newcommand{\NAE}{\mathrm{NAE}}
\newcommand{\bits}{\{0,1\}}

\begin{document}

\title[The Scaling Window of Random $k$-SAT]
{The Scaling Window of Random $k$-SAT}

\author[G. Carenini]{Gaia Carenini}
\address{Department of Pure Mathematics and Mathematical Statistics, University of Cambridge, Cambridge, United Kingdom}
\email{gc645@cam.ac.uk}

\begin{abstract}
We prove a general upper bound for scaling windows of sparse monotone covering problems. From that, we deduce that for every fixed value $k\geq 3$, the  window of random $k$-SAT is 
$O(n/\log n)$, improving the Friedgut–Bourgain bound of $O(n/\log\log n)$. We also show that random signed Not-All-Equal-$k$-SAT and hypergraph non-two-colourability (Property B) have windows of the same order. Our upper bound for scaling windows of sparse monotone covering problems follows from a general strategy combining the strengthening of  Bourgain's sharp threshold theorem by Keevash, Lifshitz, Long, and Minzer with a local-to-random replacement principle. This provides a systematic method for proving scaling-window bounds, addressing a question of Perkins.
\end{abstract}

\maketitle

\section{Introduction}

 Random discrete structures typically undergo phase transitions.  That is, a relatively small change in a parameter can cause a swift change in the structure of the overall object.
This
basic phenomenon---first understood in detail for the emergence of the giant component in
Erd\H{o}s--R\'enyi graphs~\cite{ErdosRenyiEvolution}---has since been found to organize an enormous
range of models, from random graphs and hypergraphs to random constraint satisfaction problems,
random matrices, and statistical-physics models on random and non-random lattices. Once a
transition is known to occur, three further questions present themselves. Where, as a function of
the size of the structure, does it occur? Is it sharp, in the sense that the property goes from
unlikely to likely over a window that is small compared to the location of the transition? And,
quantitatively, how small is that window?

The first two questions now have general theories attached to them. Locating a transition is often
approachable by first- and second-moment methods, or by more refined methods from statistical
physics. Meanwhile, a substantial body of work, beginning with the influence theory of Kahn,
Kalai, and Linial~\cite{KKL} and developed through the sharp-threshold machinery of Bourgain,
Friedgut, Hatami, and others~\cite{Friedgut, AP, Hatami,ODonnell}, gives general sufficient conditions
for a monotone property to have a sharp threshold. A major recent development in this direction was the proof by Park and Pham~\cite{ParkPham} of the expectation-threshold conjecture of Kahn and Kalai~\cite{KahnKalai}. Their work builds on the fractional expectation-threshold theorem of Frankston, Kahn, Narayanan, and Park~\cite{FrankstonKahnNarayananPark}, which in turn developed ideas arising from the breakthrough of Alweiss, Lovett, Wu, and Zhang on the sunflower conjecture~\cite{AlweissLovettWuZhang}. The Park-Pham theorem shows, in wide generality, that the location of a threshold is determined, up to a logarithmic factor, by a simple combinatorial covering parameter.

The third question---the width of the window, once sharpness is known---has no comparable general
theory. As Perkins emphasizes in his survey of the area~\cite{PerkinsSurvey}, sharp-threshold
theorems are typically \emph{qualitative}: they certify that a window is small compared to the
threshold location, without producing any usable quantitative bound on how small. The window
bounds that do exist in the literature have generally come from arguments tailored to the fine
structure of one specific model. The most prominent example is the essentially complete
understanding of the random $2$-SAT scaling window due to Bollob\'as, Borgs, Chayes, Kim, and
Wilson~\cite{BollobasEtAlTwoSAT}. Their analysis relies on a detailed combinatorial description of
the transition. However, such a description is unavailable once $k\ge3$.

Random $k$-SAT is the standard example of the gap between what sharp-threshold theory gives and what one would like to know. Friedgut, with an appendix by Bourgain, showed that random $k$-SAT has a sharp threshold for every fixed $k\ge3$~\cite{Friedgut}---one of the cornerstone results of the general theory. Sharpness has therefore been known since 1999, but turning the qualitative statement of Friedgut and Bourgain into an explicit window bound produces only $O(n/\log\log n)$ clauses~\cite{BalBourgain,PerkinsSurvey}, and this bound has not been improved for more than twenty-five years.

Abbe and Montanari~\cite{AbbeMontanari} showed that a transition window of order
$n/\log^{1+\delta}n$, for any fixed $\delta>0$, would have major consequences for the
satisfiability conjecture, a long-standing conjecture that asserts that, for every fixed $k$, there is an
asymptotically well-defined critical clause density at which a random $k$-CNF formula transitions
from being satisfiable to being unsatisfiable with high probability; see \cite{DingSlySun} for an overview of the state of the art on the conjecture.   Thus, there is real motivation for obtaining a better bound and, more
importantly for the present paper, for developing a method that does not depend on the special
combinatorics of the fully understood $k=2$ case.

This paper supplies such a technique. We prove an abstract window theorem, stated in
\cref{sec:prelim-overview} and proved in \cref{sec:abstract-window}, which converts a
sharp-threshold-type input into an explicit window bound for a monotone event. It has
two hypotheses, which can be verified separately in applications: short minimal witnesses for the
event must be rare, and the deterministic probability boost supplied by a strong quantitative
form of Bourgain's theorem must be reproducible, with comparable strength, by adding genuinely
random coordinates. 

The sharp additive input we use is the recent hypercontractivity theorem for global functions of
Keevash, Lifshitz, Long, and Minzer~\cite{KLLM}. The simulation step---turning a small
deterministic set of helpful coordinates into a comparably effective random one---is new; it is what allows the argument to be separated from the specific combinatorics of
any one model. This mechanism directly addresses a question posed in Perkins's survey, which asks
for a systematic method for obtaining quantitative window bounds~\cite{PerkinsSurvey}.

Applied to random $k$-SAT, the abstract theorem gives a clause window of order
$O\left({n}/{\log n}\right)$
for every fixed $k\ge3$. This improves, by a factor of order
$\log n/\log\log n$, the bound that has stood since the original Friedgut–Bourgain argument. We want to
be precise about what this does and does not establish: it is an improved \emph{upper} bound on the
window, not a matching lower bound, and it falls short of the Abbe--Montanari scale needed to
obtain consequences for the satisfiability conjecture. We discuss in \cref{sec:concluding} why the
present method, run as far as it can go, produces exactly one power of $\log n$ and no more, and
what would be needed to do better.

The $k$-SAT bound is in fact a special case of a bound for general atomic constraint satisfaction
problems (CSPs) over an arbitrary finite alphabet, stated in \cref{sec:atomic-csps}. As further applications, we also consider random signed Not-All-Equal-$k$-SAT, and  non-two-colourability of
random $k$-uniform hypergraphs (Property B). In signed NEA-$k$-SAT, each clause has its own pattern of literal negations and therefore forbids an arbitrary antipodal pair $\{a, \bar a\}$ of local assignments. Property B is the unsigned analogue: the variables represent the two vertex colours, and every hyperedge forbids only two monochromatic assignments $0^k$ and $1^k$. Although an individual signed NAE clause can be made unsigned by flipping some variables, these flips need to be consistent across overlapping clauses, so the two models are inherently different. For both problems, we prove an $O(n/\log n)$ window bound. 

Random NAE-$k$-SAT is one of
the standard benchmark models in random CSP theory. Its global complement symmetry makes it
structurally different from ordinary $k$-SAT, and it has played an important role both in the
study of satisfiability thresholds~\cite{AchlioptasMooreAsymptotic,CojaOghlanPanagiotou,DingSlySunNAE}
and in the study of condensation and the geometry of the solution
space~\cite{SlySunZhang,NamSlySohnI,NamSlySohnII,SlySohnLocal}. Hypergraph two-colorability, traditionally known as Property B, is one of the oldest models in the area--Erdős's original probabilistic argument for Property B predates the modern threshold literature--and it remains a standard testbed for hypergraph covering thresholds.  Treating both models therefore proves that the method applies beyond atomic CSPs while accommodating genuinely different constraint symmetries. 

Although these three models are closely related, they have different local symmetries; thus, the local replacement argument takes a different form in each case. Sections~\ref{sec:fibre-interfaces} and~\ref{sec:finite-predicate} isolate the common part of these
arguments and package the model-specific verification into a finite, checkable criterion based on
a rigidity certificate and a comparison between the locally induced constraints and the original
random constraints.

\section{Preliminaries and overview of the main theorem}
\label{sec:prelim-overview}

Let $\Omega_n$ be a finite coordinate set, $N_n=|\Omega_n|$, and let
$\F_n\subseteq 2^{\Omega_n}$ be a nontrivial monotone event.  We write $\mu_p$ for product
measure on $2^{\Omega_n}$, and for $0<\alpha<1$ define
$$
        p_\alpha=p_\alpha(n)
        :=\inf\{p\in[0,1]:\mu_p(\F_n)\ge\alpha\}.
$$
Since $p\mapsto\mu_p(\F_n)$ is continuous, $\mu_{p_\alpha}(\F_n)=\alpha$.
For $0<\eta<1/2$, define the central window, measured in expected number of coordinates, by
$$
        W_\eta(n)=N_n(p_{1-\eta}-p_\eta).
$$
We also define the central expected-size and variance scales
$$
        M_\eta(n)=\sup_{p\in[p_\eta,p_{1-\eta}]}pN_n,
        \qquad
        V_\eta(n)=\inf_{p\in[p_\eta,p_{1-\eta}]}N_np(1-p).
$$

The product-measure formulation is the natural setting for influence inequalities.  Random CSPs
are more commonly presented with an exact number of independently sampled constraints.  In order to keep the differences between these models out of the applications, we outline a generic transfer principle.

For $0\le m\le N_n$, let $\nu_m$ be the uniform measure on
$\binom{\Omega_n}{m}$, and let
$a_m=\nu_m(\F_n)$.
The sequence $(a_m)$ is nondecreasing.  Define
$$
        m_\alpha=\min\{m:a_m\ge\alpha\},
        \qquad
        \widehat W_\eta(n)=m_{1-\eta}-m_\eta.
$$

\begin{prop}[Transfer from product measure to uniform layers]
\label[prop]{prop:bernoulli-layer-transfer}
For every fixed $0<\eta<1/2$, and all sufficiently large $n$,
$$
        \widehat W_\eta(n)
        \le
        W_{\eta/4}(n)
        +C_\eta\sqrt{M_{\eta/4}(n)+1}.
$$
Consequently, if $M_{\eta/4}(n)=O(n)$ and
$W_{\eta/4}(n)=O(n/\log n)$, then
$\widehat W_\eta(n)=O(n/\log n)$.
\end{prop}

\begin{proof}
Let
$p_-=p_{\eta/4}$, and 
$p_+=p_{1-\eta/4}$,
and let $B_\pm\sim\operatorname{Bin}(N_n,p_\pm)$.  Write
$\sigma_\pm^2=N_np_\pm(1-p_\pm)$.  Choose $t=t(\eta)$ so that
$t^{-2}\le\min\{1/2,\eta/4\}$, and set
$r_-=\left\lfloor N_np_--t(\sigma_-+1)\right\rfloor$, and $r_+=\left\lceil N_np_++t(\sigma_++1)\right\rceil$.
Since $\mu_p(\F_n)=\mathbb E a_{\operatorname{Bin}(N_n,p)}$ and
$(a_m)$ is nondecreasing, Chebyshev's inequality gives
$$
\begin{aligned}
        \frac{\eta}{4}
        =\mathbb E a_{B_-}
        &\ge a_{r_-}\Prob(B_-\ge r_-)
        \ge \frac12 a_{r_-},\\
        1-\frac{\eta}{4}
        =\mathbb E a_{B_+}
        &\le a_{r_+}+\Prob(B_+>r_+)
        \le a_{r_+}+\frac{\eta}{4}.
\end{aligned}
$$
Thus $a_{r_-}<\eta$ and $a_{r_+}>1-\eta$, whence
$m_\eta>r_-$ and $m_{1-\eta}\le r_+$.  Therefore
$$
\begin{aligned}
        \widehat W_\eta(n)
        &\le N_n(p_+-p_-)
          +C_\eta(\sigma_-+\sigma_++1)\\
        &\le W_{\eta/4}(n)
          +C_\eta\sqrt{M_{\eta/4}(n)+1},
\end{aligned}
$$
after changing the constant.
\end{proof}

Let $\widetilde\nu_m$ be the law obtained by drawing
$Q_1,\ldots,Q_m$ independently and uniformly from $\Omega_n$, and retaining the set of
distinct sampled coordinates.  This is the standard with-replacement model for random
constraints.  Write $\widetilde a_m=\widetilde\nu_m(\F_n)$, define
$\widetilde m_\alpha$ analogously, and let
$\widetilde W_\eta=\widetilde m_{1-\eta}-\widetilde m_\eta$.

\begin{prop}[Transfer to independently sampled constraints]
\label[prop]{prop:layer-replacement-transfer}
Fix $0<\eta<1/2$.  Suppose that, throughout the relevant central range,
$m=O(n)$, and that $N_n=\Omega(n^2)$.  Then
$$
        \widetilde W_\eta(n)
        \le \widehat W_{\eta/4}(n)+O_\eta(1).
$$
In particular, an $O(n/\log n)$ uniform-layer window gives the same bound in the
with-replacement model.
\end{prop}

\begin{proof}
Let $S_m=\{Q_1,\ldots,Q_m\}$, $L_m=|S_m|$, and $D_m=m-L_m$.  Conditional on
$L_m=\ell$, the set $S_m$ is uniform on $\binom{\Omega_n}{\ell}$.  Moreover,
$$
        \mathbb E D_m
        \le \sum_{1\le i<j\le m}\Prob(Q_i=Q_j)
        =\frac{\binom m2}{N_n}
        =O(1)
$$
uniformly in the central range.  Hence, for a constant $T=T(\eta)$,
$\Prob(D_m>T)\le\eta/4$.

If $m<m_{\eta/4}$, then $L_m\le m$ and monotonicity gives
$\widetilde a_m\le a_m<\eta$.  On the other hand, for
$m=m_{1-\eta/4}+T$, the event $D_m\le T$ implies
$L_m\ge m_{1-\eta/4}$, and therefore
$$
        \widetilde a_m
        =\mathbb E a_{L_m}
        \ge \left(1-\frac{\eta}{4}\right)\Prob(D_m\le T)
        \ge\left(1-\frac{\eta}{4}\right)^2
        >1-\eta.
$$
It follows that
$\widetilde m_\eta\ge m_{\eta/4}$ and
$\widetilde m_{1-\eta}\le m_{1-\eta/4}+T$, proving the claim.
\end{proof}

\begin{cor}[Linear-scale model transfer]
\label[cor]{cor:linear-model-transfer}
Suppose $N_n=\Omega(n^2)$, the central expected size is $\Theta(n)$, and the Bernoulli
window is $O(n/\log n)$ at every fixed central level.  Then the same $O(n/\log n)$ bound
holds in the uniform fixed-size model and in the model with independently sampled coordinates.
\end{cor}

\begin{proof}
Apply \cref{prop:bernoulli-layer-transfer,prop:layer-replacement-transfer}.  The
$O(\sqrt n)$ and $O_\eta(1)$ transfer errors are $o(n/\log n)$.
\end{proof}

Assume that $\F_n$ is generated by a family $\W_n$ of inclusion-minimal witnesses, meaning
$X\in\F_n$ if and only if there exists $Y\in\W_n$ such that $Y\subseteq X$, and each $Y\in\W_n$ is inclusion-minimal with this property.

For $s\ge1$, define
$$
        \SmallN_{n,s}=
        \{X:\exists Y\in\W_n,\ |Y|\le s,\ Y\subseteq X\},
$$
and
$$
        \G_{n,s}=
        \{X:\exists Y\in\W_n,\ |Y|>s,\ Y\subseteq X\}.
$$
Then
\begin{equation}
\label{eq:G-F-small}
        \G_{n,s}\subseteq\F_n\subseteq \G_{n,s}\cup\SmallN_{n,s},
\end{equation}
and $\G_{n,s}$ is monotone.

The event $\SmallN_{n,s}$ records the presence of a short minimal witness, while
$\G_{n,s}$ is the event generated only by witnesses longer than $s$.  The inclusions
\eqref{eq:G-F-small} permit us to prune short witnesses before applying a Bourgain-type theorem and
then return to the original event with a controlled error.

A set $A\subseteq\Omega_n$ will be called a \emph{clean booster} at density $p$ and scale
$K$ if $A\notin\F_n$, $|A|=O(K)$, and
$$
        \mu_p((\F_n)_{A\to1})-\mu_p(\F_n)\ge \exp(-O(K)).
$$
The condition $A\notin\F_n$ is essential: the booster reveals useful local structure but does not
already contain a witness.

\begin{defn}[Simulable clean boosters]
\label[defn]{def:simulable}
We say that $(\F_n)$ has simulable clean boosters in the central interval if the following holds.
For every $0<\eta<1/2$ and every input constant $C_0>0$, there are constants
$C_1,c>0$ such that, whenever
$p\in[p_\eta,p_{1-\eta}]$, $1\le K\le c\log V_\eta(n)$,
and $A\subseteq\Omega_n$ satisfies $A\notin\F_n$,
$|A|\le C_0K$, and 
        $$\mu_p((\F_n)_{A\to1})
        \ge\mu_p(\F_n)+\exp(-C_0K),
$$
there exist a finite coordinate set $\Omega'$, a monotone event
$\mathcal H\subseteq2^{\Omega'}$, and an integer $D\le\exp(C_1K)$ such that
$|\Omega'|p(1-p)\ge cV_\eta(n)$.

Moreover, if $X'\sim\mu_p$ on $2^{\Omega'}$ and $R_D$ is a multiset of $D$
independent uniformly random coordinates of $\Omega'$, independent of $X'$, then
$$
        \Prob(X'\cup R_D\in\mathcal H)
        \ge\mu_p(\mathcal H)+\exp(-C_1K).
$$
\end{defn}

The second hypothesis in the theorem below is deliberately stated through an auxiliary event
$\mathcal H$.  In applications, forcing a local booster first exposes a lower-dimensional fibre,
and the natural random sprinkling lives on the coordinate set associated with that fibre rather
than on the original coordinate set.  The variance comparison in \cref{def:simulable} ensures that
this passage does not destroy the scale relevant to the layer estimate.

\begin{thm}
\label{thm:abstract-window}
Fix $0<\eta<1/2$, and suppose $V_\eta(n)\to\infty$ and $p_{1-\eta}\le1/2$ for all sufficiently large $n$.  Assume that $(\F_n)$ satisfies:
\begin{enumerate}[label=(\roman*)]
\item there are constants $B,c,C>0$ such that, for every
$1\le K\le c\log V_\eta(n)$ and $s=BK$,
$$
        \sup_{p\in[p_\eta,p_{1-\eta}]}
        \mu_p(\SmallN_{n,s})\le\exp(-CK),
$$
where, for the constants in \cref{lem:clean-booster} with parameter $\eta/2$,
$B\ge B_{\eta/2}$, and 
$$C\ge \max\left\{C_{\eta/2}^{\rm wit}+\log4,\ \log(2/\eta)\right\};$$
\item $(\F_n)$ has simulable clean boosters.
\end{enumerate}
Then
$$
        W_\eta(n)\le C_\eta\frac{M_\eta(n)}{\log V_\eta(n)}.
$$
\end{thm}

The theorem is useful when $M_\eta(n)$ and $V_\eta(n)$ have comparable polynomial size.  In all
three applications of this paper they are $\Theta(n)$, so the conclusion becomes
$W_\eta(n)=O(n/\log n)$.

\subsection{Overview of the proof}

Suppose that the central window is wide.  Choose a parameter $K$ of order
$M_\eta(n)/W_\eta(n)$, and delete all minimal witnesses of size at most $s=\Theta(K)$.  The
small-witness hypothesis makes this pruning negligible throughout the central interval.  Russo's
formula then provides a density $p$ in that interval at which the pruned event has
$pI_p=O(K)$.  The sharp additive Bourgain theorem of Keevash, Lifshitz, Long, and Minzer produces a
set of $O(K)$ coordinates whose forcing raises the probability by $\exp(-O(K))$.

The first key step is the extraction of a clean booster for the original event.  On a
configuration newly brought into the pruned event, choose a minimal witness and intersect it with
the Bourgain set.  Since the witness has size greater than $s$, whereas the Bourgain set is much
smaller than $s$, this intersection is a proper sub-witness and hence cannot itself force
$\F_n$.  Pigeonholing over the possible intersections loses only another exponential factor.

The second interface is model-specific.  Simulability converts the clean deterministic booster into
at most $D=\exp(O(K))$ uniformly random coordinates for an auxiliary monotone event.  A universal
Boolean-layer estimate shows that adding $D$ random coordinates can increase the probability of
an event on $N$ coordinates by at most
$$
        O\!\left(\frac{D+1}{\sqrt{Np(1-p)}}\right).
$$
Comparing this with the $\exp(-O(K))$ boost gives
$$
        \exp(-O(K))
        \le \frac{\exp(O(K))}{\sqrt{V_\eta(n)}},
$$
and therefore $K=\Omega(\log V_\eta(n))$.  Since
$K\asymp M_\eta(n)/W_\eta(n)$, this rearranges to the claimed window bound.

For atomic CSPs, a clean booster leaves one exceptional fibres and induces random
codimension-$(k-1)$ cylinders there.  For signed NAE-SAT, it leaves two antipodal fibres and
induces complementary atomic cylinders.  For hypergraph two-colourability, it again leaves two
antipodal fibres but induces only constant cylinders.  The replacement theorems in the later
sections are precisely what verify simulability in these three geometries.

\section{Proof of the abstract window theorem}
\label{sec:abstract-window}
The first delicate step is the passage from the Bourgain-type booster supplied by
Keevash, Lifshitz, Long, and Minzer to a \emph{clean} booster, namely a small set which is not itself
a witness for $\F_n$.  We isolate this step in \cref{lem:clean-booster}.  No structure of the later
constraint-satisfaction models is used here; the only inputs are pruning by minimal witnesses and
rarity of small witnesses.

For a Boolean function $f$, let
$$
        I_p(f)=\sum_{i\in\Omega}\Prob_p(i\text{ is pivotal for }f)
$$
denote its total $p$-biased influence.  With this convention, Russo's formula reads
$\frac{d}{dp}\mu_p(f)=I_p(f)$ for monotone $f$.  We use the following theorem of Keevash,
Lifshitz, Long, and Minzer~\cite[Theorem~1.4]{KLLM}.

\begin{thm}[Sharp additive Bourgain theorem]
\label{thm:KLLM}
For every $0<\eta<1/2$ there is $C_\eta<\infty$ such that the following holds.  Let
$f:2^\Omega\to\{0,1\}$ be monotone, let $0<p\le1/2$, and suppose
$\eta\le \mu_p(f)\le 1-\eta$,
and $pI_p(f)\le K$. Then there is a set $J\subseteq\Omega$ such that
$|J|\le C_\eta K$
and
$$
        \mu_p(f_{J\to1})\ge \mu_p(f)+\exp(-C_\eta K).
$$
Here $f_{J\to1}$ is obtained by setting every coordinate of $J$ equal to $1$.
\end{thm}

\begin{lem}[Pruned clean-booster extraction]
\label[lem]{lem:clean-booster}
Fix $0<\eta<1/2$.  There are constants $B_\eta>0$,
$C_\eta^{\rm wit}>0$, and $C_\eta^{\rm cl}>0$ such that the following holds.  Let
$0<p\le1/2$, let $K\ge1$, and let $s\ge B_\eta K$.  Suppose
$\eta\le\mu_p(\G_{n,s})\le1-\eta$, $pI_p(\G_{n,s})\le K$,
and
$$
        \mu_p(\SmallN_{n,s})
        \le\frac14\exp(-C_\eta^{\rm wit}K).
$$
Then there is a nonempty set $A\subseteq\Omega_n$ such that
$A\notin\F_n$, $|A|\le C_\eta^{\rm cl}K$, and
$$
        \mu_p((\F_n)_{A\to1})
        \ge\mu_p(\F_n)+\exp(-C_\eta^{\rm cl}K).
$$
\end{lem}

\begin{proof}
Let $C_{\rm K}=C_{\rm K}(\eta)$ be the constant in \cref{thm:KLLM}.  Set
$B_\eta=2C_{\rm K}$,
$ C_\eta^{\rm wit}=C_{\rm K}$, and choose
$$
        C_\eta^{\rm cl}
        \ge (1+\log2)C_{\rm K}+\log(4/3).
$$
Apply \cref{thm:KLLM} to $\G_{n,s}$.  We obtain $J\subseteq\Omega_n$ with
$|J|\le C_{\rm K}K\le s/2$
and
$$
        \mu_p((\G_{n,s})_{J\to1})
        \ge\mu_p(\G_{n,s})+\delta,
        \qquad
        \delta=e^{-C_{\rm K}K}.
$$
Let $X\sim\mu_p$.  By monotonicity of $\G_{n,s}$, the last display is equivalent to
$$
        \Prob(X\notin\G_{n,s},\ X\cup J\in\G_{n,s})\ge\delta.
$$
Since $\mu_p(\SmallN_{n,s})\le\delta/4$,
$$
        \Prob(X\notin\G_{n,s},\ X\notin\SmallN_{n,s},\
              X\cup J\in\G_{n,s})
        \ge\frac{3\delta}{4}.
$$
On this event, $X\notin\F_n$, by \eqref{eq:G-F-small}.  Fix an arbitrary ordering of
$\W_n$.  Since $X\cup J\in\G_{n,s}$, let $Y=Y(X)$ be the first minimal witness in this
ordering such that
$Y\subseteq X\cup J$, and $|Y|>s$.
Set $A_X=Y\cap J$.  Then $A_X\ne\varnothing$, since otherwise $Y\subseteq X$, contrary
to $X\notin\F_n$.  Moreover,
$$
        |A_X|\le|J|\le s/2<|Y|,
$$
so $A_X\subsetneq Y$.  If $A_X\in\F_n$, then $A_X$ contains a minimal witness
$Y'\in\W_n$, and
$Y'\subseteq A_X\subsetneq Y$,
contradicting the inclusion-minimality of $Y$.  Thus $A_X\notin\F_n$.  Finally,
$$
        Y=(Y\cap X)\cup(Y\cap J)\subseteq X\cup A_X,
$$
so $X\cup A_X\in\F_n$.

There are at most
$2^{|J|}\le\exp(C_{\rm K}K\log2)$
possible values of $A_X$.  Hence some fixed nonempty $A\subseteq J$, with
$A\notin\F_n$, satisfies
$$
\begin{aligned}
        \Prob(X\notin\F_n,\ X\cup A\in\F_n)
        &\ge \frac{3}{4}
              \exp(-(1+\log2)C_{\rm K}K)\\
        &\ge \exp(-C_\eta^{\rm cl}K).
\end{aligned}
$$
Since $|A|\le C_{\rm K}K\le C_\eta^{\rm cl}K$, and since for a monotone event
$$
        \mu_p((\F_n)_{A\to1})-\mu_p(\F_n)
        =\Prob(X\notin\F_n,\ X\cup A\in\F_n),
$$
the claim follows.
\end{proof}

\begin{rem}
\label[rem]{rem:kllm-clean-interface}
Lemma~\ref{lem:clean-booster} is where Theorem~\ref{thm:KLLM} is converted into a statement
about the original event $\F_n$.  The pruning parameter $s$ is chosen much larger than the
booster size $|J|$, so that the part of a new large witness lying inside $J$ is a proper
sub-witness and hence cannot already force $\F_n$.  This is the mechanism that produces the
clean condition $A\notin\F_n$.
\end{rem}

\begin{lem}[Random-coordinate layer increment]
\label[lem]{lem:layer-increment}
Let $\mathcal H\subseteq2^\Omega$ be monotone, with $|\Omega|=N$, let $0<p<1$, and let
$D\ge0$ be an integer.  Let $X\sim\mu_p$, and obtain $Y$ from $X$ by adding $D$
independent uniformly random coordinates of $\Omega$, independent of $X$, with repetitions
allowed and already-present coordinates ignored.  Then
$$
        \Prob(Y\in\mathcal H)-\Prob(X\in\mathcal H)
        \le
        C\min\left\{1,\frac{D+1}{\sqrt{Np(1-p)}}\right\}.
$$
\end{lem}

\begin{proof}
For $0\le r\le N$, define the layer density
$$
        a_r=\frac{|\mathcal H\cap\binom{\Omega}{r}|}{\binom Nr}.
$$
Since $\mathcal H$ is monotone, the normalized matching property of the Boolean lattice implies
that $(a_r)_{r=0}^N$ is nondecreasing.

Let $S=|X|\sim\Bin(N,p)$.  Conditional on $S=s$, the set $X$ is uniform on the
$s$-th layer.  After adding $D$ coordinates, the resulting set has size at most $s+D$; conditional on its
final size it is uniform on the corresponding layer, by symmetry.  Thus, by monotonicity and
averaging over layers,
$$
        \Prob(Y\in\mathcal H\mid S=s)-\Prob(X\in\mathcal H\mid S=s)
        \le a_{s+D}-a_s,
$$
where $a_t=1$ for $t>N$.  Therefore
$$
        \Prob(Y\in\mathcal H)-\Prob(X\in\mathcal H)
        \le \E[a_{S+D}-a_S].
$$
For completeness, put $\Delta_j=a_j-a_{j-1}\ge0$, extending $a_j$ by $0$ below
$0$ and by $1$ above $N$.  Then
$$
\begin{aligned}
        \E[a_{S+D}-a_S]
        &=\sum_j\Delta_j\Prob(j-D\le S\le j-1)\\
        &\le \sup_t\Prob(S\in[t-D,t])\sum_j\Delta_j\\
        &\le \sup_t\Prob(S\in[t-D,t]).
\end{aligned}
$$
The standard binomial local bound gives
$$
        \sup_t\Prob(S=t)\le \frac{C}{\sqrt{Np(1-p)}}.
$$
Summing over at most $D+1$ consecutive values gives the claim, together with the trivial bound by
$1$.
\end{proof}

We now combine pruning, the clean-booster extraction, simulability, and the layer estimate.

\begin{proof}[Proof of \cref{thm:abstract-window}]
Apply \cref{lem:clean-booster} with parameter $\eta/2$, and write its constants as
$B_{\eta/2}$, $C_{\eta/2}^{\rm wit}$, and $C_{\eta/2}^{\rm cl}$.  Apply
\cref{def:simulable} with input constant $C_0=C_{\eta/2}^{\rm cl}$, and denote the resulting
constants by $C_{\rm sim}$ and $c_{\rm sim}$.  Replacing $c$ by a smaller constant if
necessary, we may assume $c\le c_{\rm sim}$.  By hypothesis~(i),
$B\ge B_{\eta/2}$, $C\ge C_{\eta/2}^{\rm wit}+\log4$, and
$e^{-C}\le\eta/2$.

Assume for contradiction that
$$
        W_\eta(n)\ge L\frac{M_\eta(n)}{\log V_\eta(n)}
$$
for a constant $L=L_\eta$ to be chosen.  Put
$$
        K_0=\max\left\{1,\frac{M_\eta(n)}{W_\eta(n)}\right\}.
$$
Then
$$
        K_0\le1+\frac{1}{L}\log V_\eta(n).
$$
Taking $L$ sufficiently large and then $n$ sufficiently large, we have
$K_0\le c\log V_\eta(n)$.  Let $s=BK_0$.

Hypothesis~(i), with $K=K_0$, gives
$$
        \sup_{p\in[p_\eta,p_{1-\eta}]}\mu_p(\SmallN_{n,s})
        \le e^{-CK_0}\le e^{-C}\le\eta/2.
$$
For every $p\in[p_\eta,p_{1-\eta}]$, \eqref{eq:G-F-small} therefore implies
$$
\begin{aligned}
        \mu_p(\G_{n,s})
        &\ge \mu_p(\F_n)-\mu_p(\SmallN_{n,s})\ge\eta/2,\\
        \mu_p(\G_{n,s})
        &\le \mu_p(\F_n)\le1-\eta\le1-\eta/2.
\end{aligned}
$$
Russo's formula gives
$$
        \int_{p_\eta}^{p_{1-\eta}}I_p(\G_{n,s})\,dp
        =\mu_{p_{1-\eta}}(\G_{n,s})-\mu_{p_\eta}(\G_{n,s})
        \le1.
$$
Since $p_{1-\eta}-p_\eta=W_\eta(n)/N_n$, there is
$p\in[p_\eta,p_{1-\eta}]$ such that
$$
        I_p(\G_{n,s})\le \frac{N_n}{W_\eta(n)}.
$$
Consequently,
$$
        pI_p(\G_{n,s})
        \le \frac{pN_n}{W_\eta(n)}
        \le \frac{M_\eta(n)}{W_\eta(n)}
        \le K_0.
$$
By \cref{lem:clean-booster}, there is a nonempty $A\subseteq\Omega_n$ such that
$A\notin\F_n$,
$|A|\le C_{\eta/2}^{\rm cl}K_0$, and $$
 \mu_p((\F_n)_{A\to1})
        \ge\mu_p(\F_n)+e^{-C_{\eta/2}^{\rm cl}K_0}.
$$
The simulability hypothesis now gives a finite set $\Omega'$, a monotone event
$\mathcal H\subseteq2^{\Omega'}$, and an integer $D\le e^{C_{\rm sim}K_0}$ such that
$|\Omega'|p(1-p)\ge c_{\rm sim}V_\eta(n)
$
and
$$
        \Prob(X'\cup R_D\in\mathcal H)-\mu_p(\mathcal H)
        \ge e^{-C_{\rm sim}K_0}.
$$
By \cref{lem:layer-increment},
$$
        e^{-C_{\rm sim}K_0}
        \le C_{\rm lay}\frac{D+1}{\sqrt{|\Omega'|p(1-p)}}
        \le \frac{C_\eta e^{C_{\rm sim}K_0}}{\sqrt{V_\eta(n)}}.
$$
Taking logarithms, there is a constant $c_\eta'>0$ such that
$K_0\ge c_\eta'\log V_\eta(n)$
for all sufficiently large $n$.  On the other hand,
$$
        K_0\le1+\frac{1}{L}\log V_\eta(n).
$$
Choosing $L>2/c_\eta'$ and then $n$ large yields a contradiction.  This proves the claimed
bound with $C_\eta=L$.
\end{proof}

\begin{rem}
\label[rem]{rem:abstract-scope}
The hypothesis $p_{1-\eta}\le1/2$ is included only to match the form of Theorem~\ref{thm:KLLM}.  In the atomic CSP application this condition follows from the linear central
scale, since the relevant probabilities satisfy $p=\Theta(n^{1-k})$.  Note that no symmetry or CSP structure
is used in the abstract theorem beyond the two stated hypotheses.
\end{rem}

\section{Exceptional fibres and two local-to-random reduction}
\label{sec:fibre-interfaces}

The three applications below have different local geometries, but the probabilistic part of their
fibre arguments is identical.  We isolate it here in two results.  The first converts a binomial
family of helpful lower-arity objects into a fixed number of independent uniform helpers.  The
second converts a pointwise power comparison between two classes of local objects into an averaged
replacement theorem which preserves a positive excess probability.

Let $X$ be a finite set and let $\mathcal Q$ be a finite family of subsets of $X$.  For a
subfamily $\mathcal S\subseteq\mathcal Q$, write
$\bigcup\mathcal S=\bigcup_{Q\in\mathcal S}Q$.  A $\theta$-random subfamily of
$\mathcal Q$ is obtained by retaining each member independently with probability $\theta$.

\begin{lem}[Binomial-to-independent helper extraction]
\label[lem]{lem:binomial-helper-extraction}
Let $B$ be a random subset of $X$, and let $\mathcal H_\theta$ be a
$\theta$-random subfamily of $\mathcal Q$, independent of $B$.  Let
$\zeta=\Prob(B=\varnothing)$, and $\lambda=\theta|\mathcal Q|$.
Suppose that, for some $K\ge1$, $\Lambda\ge1$, and $0<\beta\le1/2$,
$\lambda\le\Lambda K$, and$\Prob\!\left(B\subseteq\bigcup\mathcal H_\theta\right)
        \ge\zeta+\beta$.
There is a constant $C=C(\Lambda)$ such that, 
if
$$
        |\mathcal Q|\ge \frac{4\left\lceil C\bigl(K+\log(1/\beta)\bigr)\right\rceil^2}{\beta},
$$
then there is an integer $1\le d\le \left\lceil C\bigl(K+\log(1/\beta)\bigr)\right\rceil$ for which independent uniformly random
$Q_1,\ldots,Q_d\in\mathcal Q$, also independent of $B$, satisfy
$$
        \Prob(B\subseteq Q_1\cup\cdots\cup Q_d)
        \ge\zeta+\frac\beta2.
$$
\end{lem}

\begin{proof}
Let $L=|\mathcal H_\theta|\sim\operatorname{Bin}(|\mathcal Q|,\theta)$, and let $T=\left\lceil C\bigl(K+\log(1/\beta)\bigr)\right\rceil$.  Choose the
constant in the definition of $T$ so that the standard binomial upper-tail estimate gives
$$
        \Prob(L>T)\le\frac\beta8.
$$
Indeed, for $T\ge e^2\lambda$,
$$
        \Prob(L\ge T)
        \le \left(\frac{e\lambda}{T}\right)^T
        \le e^{-T},
$$
and the asserted choice of $T$ ensures both $T\ge e^2\lambda$ and
$e^{-T}\le\beta/8$.

Conditional on $L=\ell$, the family $\mathcal H_\theta$ is a uniformly random
$\ell$-element subset of $\mathcal Q$.  Let
$$
        b_\ell
        =\Prob\!\left(B\subseteq\bigcup\mathcal S_\ell\right)-\zeta,
$$
where $\mathcal S_\ell$ is such a subset, independent of $B$.  Since adding sets can only
help, $b_\ell\ge0$, and $b_0=0$.  Therefore
$$
\begin{aligned}
        \beta
        &\le \sum_{\ell\ge0}\Prob(L=\ell)b_\ell\\
        &\le \sum_{1\le\ell\le T}\Prob(L=\ell)b_\ell
              +\Prob(L>T).
\end{aligned}
$$
Thus some $1\le d\le T$ satisfies $b_d\ge7\beta/8$.

A uniformly random $d$-element subset of $\mathcal Q$ has the same law as
$d$ independent uniform samples conditioned on having no collision.  The total variation
distance between these two laws is at most
$$
        \Prob(\text{a collision})
        \le\frac{\binom d2}{|\mathcal Q|}
        \le\frac\beta8.
$$
Consequently the independent sample has excess at least
$7\beta/8-\beta/8\ge\beta/2$, as required.
\end{proof}

\begin{cor}[Logarithmic-scale helper extraction]
\label[cor]{cor:logarithmic-helper-extraction}
Fix positive constants $C_1,C_2,a$.  Suppose
$\lambda\le C_1K$,
$\beta\ge e^{-C_2K}$,
$|\mathcal Q|\ge n^a$,
and $ 1\le K\le c\log n$.
If $c=c(C_1,C_2,a)>0$ is sufficiently small, then the conclusion of
\cref{lem:binomial-helper-extraction} holds with
$1\le d\le C K$ and excess at least $e^{-C K}$,
where $C=C(C_1,C_2,a)$.
\end{cor}

\begin{proof}
The parameter $T$ in \cref{lem:binomial-helper-extraction} is $O(K)$.  Moreover,
$4T^2/\beta\le C K^2e^{C_2K}\le n^a$ when $K\le c\log n$ and $c$ is sufficiently
small.  Apply the lemma and adjust constants.
\end{proof}

Let $\mathcal Q$ and $\mathcal R$ be finite families of subsets of the same finite set
$X$.  For $S\subseteq X$, define
$\rho_{\mathcal Q}(S)
        =\Prob_{Q\in\mathcal Q}(S\subseteq Q)$, and $\rho_{\mathcal R}(S)
        =\Prob_{R\in\mathcal R}(S\subseteq R)$,
where the objects are chosen uniformly.

\begin{thm}[Averaged power replacement]
\label{thm:averaged-power-replacement}
Let $\gamma\ge1$, $c_*>0$, and $\varepsilon_0\ge0$.  Suppose that
\begin{equation} \label{eq:pointwise-power-replacement}
 \rho_{\mathcal Q}(S)\ge\varepsilon\ge\varepsilon_0
        \quad\Longrightarrow\quad
        \rho_{\mathcal R}(S)\ge c_*\varepsilon^\gamma
\end{equation}
for every $S\subseteq X$.  Let $B$ be a random subset of $X$, independent of all local
objects below, and let $\zeta\in[0,1]$.  Suppose that $d\ge1$ independent uniform members of
$\mathcal Q$ cover $B$ with probability at least $\zeta+\alpha$, where $\alpha>0$, and
that
$\varepsilon:=\frac{\alpha}{4d}\ge\varepsilon_0.
$
Then there is an integer
$D\le C_{c_*,\gamma}
        d^{1+\gamma}\alpha^{-\gamma}\log(2d/\alpha)$
such that $D$ independent uniformly random members of $\mathcal R$ cover $B$ with
probability at least $\zeta+\alpha/2$.
\end{thm}

\begin{proof}
Let $Q_1,\ldots,Q_d$ be the original independent members of $\mathcal Q$.  Replace them one
at a time by independent blocks of members of $\mathcal R$, using fresh randomness at every
step.  At a given step, condition on $B$, on all local objects other than the $Q_i$ currently
being replaced, and on every block inserted at earlier steps.  Under this conditioning, $Q_i$
remains independent and uniform.  Let $S\subseteq B$ be the residual set not covered by the
other local objects.  The current conditional success probability is
$\rho_{\mathcal Q}(S)$.

If $\rho_{\mathcal Q}(S)<\varepsilon$, deleting $Q_i$ and inserting an arbitrary fresh block
loses less than $\varepsilon$.  Otherwise, \eqref{eq:pointwise-power-replacement} gives
$\rho_{\mathcal R}(S)\ge c_*\varepsilon^\gamma$.  Hence a block of
$g=\left\lceil C_{c_*,\gamma}\varepsilon^{-\gamma}
              \log(1/\varepsilon)\right\rceil$ 
independent members of $\mathcal R$ fails to contain $S$ with probability at most
$\varepsilon$.  In either case the replacement loses at most $\varepsilon$ in conditional,
and therefore unconditional, success probability.

After all $d$ replacements, the total loss is at most
$d\varepsilon=\alpha/4$.  The final success probability is consequently at least
$\zeta+3\alpha/4$, which is stronger than claimed.  Finally, $D=dg$, and substituting
$\varepsilon=\alpha/(4d)$ gives the displayed bound after adjusting the constant.
\end{proof}

\begin{prop}[Exceptional-fibre criterion for simulability]
\label[prop]{prop:exceptional-fibre-simulability}
Fix $0<\eta<1/2$.  Suppose that, for every input constant $C_0>0$, there are positive
constants
$C, c, c_{\rm var}, c_*, C_{\rm th}, b$, and $\gamma\ge1$
with the following property.  Whenever $p\in[p_\eta,p_{1-\eta}]$,
$1\le K\le c\log V_\eta(n)$,
and $A\subseteq\Omega_n$ is a clean booster satisfying
$A\notin\F_n$, $|A|\le C_0K$, and $\mu_p((\F_n)_{A\to1})\ge\mu_p(\F_n)+e^{-C_0K}$,
there exist the following:
\begin{enumerate}[label=(\roman*)]
\item a finite state space $X$, a finite coordinate set $\Omega'$, and subsets
$(R_\omega)_{\omega\in\Omega'}$ of $X$, defining the monotone covering event
$$
        \mathcal H
        =\left\{Y\subseteq\Omega':X\subseteq\bigcup_{\omega\in Y}R_\omega\right\};
$$
\item a finite helper family $\mathcal Q$ of subsets of $X$, and an integer
$1\le d\le CK$, such that, if $X'\sim\mu_p$ on $2^{\Omega'}$,
$B=B(X'):=X\setminus\bigcup_{\omega\in X'}R_\omega$,
and $Q_1,\ldots,Q_d$ are independent uniformly random members of $\mathcal Q$, independent
of $X'$, then
$\Prob(B\subseteq Q_1\cup\cdots\cup Q_d)
        \ge \mu_p(\mathcal H)+e^{-CK}$;
\item the pointwise comparison
\begin{equation} \label{eq:fibre-native-comparison}
 \rho_{\mathcal Q}(S)\ge\varepsilon\ge C_{\rm th}V_\eta(n)^{-b}
        \quad\Longrightarrow\quad
        \rho_{\Omega'}(S)\ge c_*\varepsilon^\gamma   
\end{equation}

for every $S\subseteq X$, where
$\rho_{\Omega'}(S)
        =\frac{1}{|\Omega'|}\bigl|\{\omega\in\Omega':S\subseteq R_\omega\}\bigr|$;
\item the variance comparison
$|\Omega'|p(1-p)\ge c_{\rm var}V_\eta(n)$.
\end{enumerate}
Then $(\F_n)$ has simulable clean boosters in the central interval.
\end{prop}

\begin{proof}
Put $\alpha=e^{-CK}$ and $\varepsilon=\alpha/(4d)$.  Since $d\le CK$, after decreasing
$c$ as a function of $C$ and $b$, the restriction
$K\le c\log V_\eta(n)$ ensures
$$
        \varepsilon
        \ge \frac{e^{-CK}}{4CK}
        \ge C_{\rm th}V_\eta(n)^{-b}.
$$
Apply \cref{thm:averaged-power-replacement} to the helper family $\mathcal Q$ and the indexed
native family $(R_\omega)_{\omega\in\Omega'}$, regarded with its $\Omega'$-multiplicity, using
\eqref{eq:fibre-native-comparison}.  We obtain an integer
$$
\begin{aligned}
        D
        &\le C_{c_*,\gamma}
        d^{1+\gamma}\alpha^{-\gamma}\log(2d/\alpha) \le \exp(C'K)
\end{aligned}
$$
for which $D$ independent uniform coordinates
$\omega_1,\ldots,\omega_D\in\Omega'$ satisfy
$$
        \Prob\!\left(B\subseteq
              R_{\omega_1}\cup\cdots\cup R_{\omega_D}\right)
        \ge\mu_p(\mathcal H)+\frac\alpha2.
$$
The event inside the probability is exactly
$X'\cup\{\omega_1,\ldots,\omega_D\}\in\mathcal H$.  Since
$\alpha/2\ge e^{-C''K}$ for $K\ge1$, and condition~(iv) gives the required variance
comparison, this is precisely \cref{def:simulable}.
\end{proof}

\begin{rem}
\label[rem]{rem:two-fibre-interfaces}
In each application, the maximal local quasi-cover and the exceptional fibres are model-specific.
Once they have been identified, however, the remainder is formal.  The selected constraints
meeting the booster variables once form a binomial family of lower-arity helpers, to which
\cref{cor:logarithmic-helper-extraction} applies.  A model-specific pointwise estimate of the form
\eqref{eq:pointwise-power-replacement} then invokes \cref{thm:averaged-power-replacement} and
produces the random native constraints required by simulability.
\end{rem}

\section{A finite criterion for forbidden-pattern systems}
\label{sec:finite-predicate}

We now turn the exceptional-fibre mechanism into a predicate-level theorem.  The criterion below
has a finite combinatorial part, which identifies the exceptional fibres, and an analytic part,
which compares one-coordinate derivatives with native constraints.

Fix a finite alphabet $\Sigma$, an integer $k\ge2$, and a nonempty finite family
$\mathscr P$ of nonempty proper subsets of $\Sigma^k$.  We assume that $\mathscr P$ is
closed under permutations of the $k$ coordinates.  A $\mathscr P$-constraint on a
$k$-set $S=\{i_1<\cdots<i_k\}$ is a pair $(S,P)$, with $P\in\mathscr P$, and it
covers the assignments $x\in\Sigma^n$ for which
$(x_{i_1},\ldots,x_{i_k})\in P$.
Let
$$
        \Omega_n(\mathscr P)=\binom{[n]}k\times\mathscr P,
        \qquad
        N_n(\mathscr P)=|\mathscr P|\binom nk,
$$
and let $\F_n(\mathscr P)$ be the event that the selected constraints cover all of
$\Sigma^n$.

Let $\Gamma\le\operatorname{Sym}(\Sigma)$ be a finite group such that every
$P\in\mathscr P$ is invariant under the diagonal action of $\Gamma$.  Define a preorder on
$\Sigma^k$ by
\begin{equation}
\label{eq:pattern-preorder}
        u\preceq_{\mathscr P}v
        \quad\Longleftrightarrow\quad
        \{P\in\mathscr P:u\in P\}
        \subseteq
        \{P\in\mathscr P:v\in P\}.
\end{equation}
For an injection $\phi:[k]\to R$, write
$x_\phi=(x_{\phi(1)},\ldots,x_{\phi(k)})$.

\begin{defn}[Rigidity]
\label[defn]{def:predicate-rigidity}
The system $(\mathscr P,\Gamma)$ is \emph{rigid} if there is an integer $L\ge k$ such that,
whenever $R$ is finite, every symbol of $\Sigma$ occurs at least $L$ times in
$\sigma\in\Sigma^R$, and $\tau\in\Sigma^R$ satisfies
\begin{equation}
\label{eq:local-rigidity-condition}
        \tau_\phi\preceq_{\mathscr P}\sigma_\phi
        \qquad\text{for every injection }\phi:[k]\to R.
\end{equation}
then $\tau=g\sigma$ for some $g\in\Gamma$.
\end{defn}

Although Definition~\ref{def:predicate-rigidity} is stated for arbitrarily large sets, it has a finite
certificate.  Let
$T_L=\Sigma\times[L]$, and $\sigma_L(a,j)=a$. For $a\in\Sigma$, let $T_{L,a}=T_L\cup\{*\}$ and extend $\sigma_L$ by
$\sigma_{L,a}(*)=a$.

\begin{prop}[Finite rigidity certificate]
\label[prop]{prop:finite-rigidity-certificate}
The system $(\mathscr P,\Gamma)$ is rigid with parameter $L$ if and only if the following two
finite assertions hold.
\begin{enumerate}[label=(\roman*)]
\item Every $\tau\in\Sigma^{T_L}$ satisfying
\eqref{eq:local-rigidity-condition} with $\sigma=\sigma_L$ is equal to
$g\sigma_L$ for some $g\in\Gamma$.
\item For every $a\in\Sigma$, $g\in\Gamma$, and
$\tau\in\Sigma^{T_{L,a}}$ satisfying \eqref{eq:local-rigidity-condition} with
$\sigma=\sigma_{L,a}$, if $\tau|_{T_L}=g\sigma_L$, then
$\tau(*)=g(a)$.
\end{enumerate}
In particular, whether a proposed $L$ is a valid rigidity parameter is decidable by a finite
computation depending only on $\Sigma,k,\mathscr P,\Gamma,L$.
\end{prop}

\begin{proof}
Necessity follows by applying rigidity to the two displayed finite templates.  Conversely, let
$\sigma\in\Sigma^R$ contain every symbol at least $L$ times, and suppose that $\tau$
satisfies \eqref{eq:local-rigidity-condition}.  Choose a subset $T\subseteq R$ containing
exactly $L$ coordinates of each $\sigma$-value.  After identifying $T$ with $T_L$,
condition~(i) gives
$\tau|_T=g\sigma|_T$ for some $g\in\Gamma$.  For each $i\in R\setminus T$, apply
condition~(ii) to $T\cup\{i\}$, with $a=\sigma_i$.  It gives
$\tau_i=g(\sigma_i)$.  Thus $\tau=g\sigma$ on all of $R$.
\end{proof}

The connection with exceptional fibres is immediate.  For $\sigma\in\Sigma^R$, let
\begin{equation}
\label{eq:maximal-pattern-quasicover}
        M(R,\sigma)
        =\{(S,P)\in\Omega_R(\mathscr P):\sigma_S\notin P\}.
\end{equation}

\begin{lem}[Rigidity gives a finite exceptional orbit]
\label[lem]{lem:rigidity-exceptional-orbit}
Suppose $(\mathscr P,\Gamma)$ is rigid and every symbol occurs at least $L$ times in
$\sigma$.  Then the assignments not covered by $M(R,\sigma)$ are exactly
$\Gamma\sigma=\{g\sigma:g\in\Gamma\}$.
\end{lem}

\begin{proof}
Every $g\sigma$ is uncovered because each forbidden set $P$ is $\Gamma$-invariant.  If
$\tau$ is uncovered, then there is no constraint $(S,P)$ for which
$\tau_S\in P$ but $\sigma_S\notin P$.  Equivalently,
$\tau_\phi\preceq_{\mathscr P}\sigma_\phi$ for every injection $\phi:[k]\to R$, where
closure of $\mathscr P$ under coordinate permutations removes the choice of ordering.  Rigidity
now gives $\tau=g\sigma$.
\end{proof}

For $a\in\Sigma$ and $P\in\mathscr P$, define the first-coordinate derivative
$\partial_aP
        =\{z\in\Sigma^{k-1}:(a,z)\in P\}$.
Let
$\mathscr D(\mathscr P)
        =\{(a,P):a\in\Sigma,\ P\in\mathscr P,\ \partial_aP\ne\varnothing\}$
be an indexed family, so equal derivative sets retain their multiplicities.  For
$S\subseteq\Sigma^m$, let $\rho_{\mathscr P,m}(S)$ be the probability that a uniformly
random native constraint $(U,P)\in\binom{[m]}k\times\mathscr P$ contains $S$, and let
$\rho_{\partial\mathscr P,m}(S)$ be the probability that a uniformly random indexed derivative
constraint
$$
        (U,a,P)\in\binom{[m]}{k-1}\times\mathscr D(\mathscr P)
$$
contains $S$, where its covered set is $\{x\in\Sigma^m:x_U\in\partial_aP\}$.

\begin{defn}[Power-replaceable derivatives]
\label[defn]{def:power-replaceable-predicate}
The system $\mathscr P$ has \emph{power-replaceable one-coordinate derivatives} if there are
constants $c_*,C_*,b>0$ and $\gamma\ge1$ such that, for every $m$ and every
$S\subseteq\Sigma^m$,
\begin{equation}
\label{eq:predicate-derivative-replacement}
        \rho_{\partial\mathscr P,m}(S)\ge\varepsilon\ge C_*m^{-b}
        \quad\Longrightarrow\quad
        \rho_{\mathscr P,m}(S)\ge c_*\varepsilon^\gamma.
\end{equation}
\end{defn}

The indexing convention in \cref{def:power-replaceable-predicate} is important.  After balancing a
local assignment, every indexed derivative has the same number of one-vertex extensions to a
native constraint.  Thus the helpers induced by random crossing constraints form an exact
binomial subfamily of the indexed derivative family.

For $s\ge1$, let $\SmallN_{n,s}(\mathscr P)$ be the event that the random
$\mathscr P$-instance contains an inclusion-minimal cover with at most $s$ constraints.

\begin{thm}[Finite-predicate logarithmic-window criterion]
\label{thm:finite-predicate-window}
Fix $0<\eta<1/2$ and a symmetric forbidden-pattern system
$(\Sigma,k,\mathscr P,\Gamma)$.  Suppose:
\begin{enumerate}[label=(\roman*)]
\item $(\mathscr P,\Gamma)$ is rigid;
\item $\mathscr P$ has power-replaceable one-coordinate derivatives;
\item uniformly for $p\in[p_\eta,p_{1-\eta}]$,
\begin{equation}
\label{eq:predicate-linear-scale}
        pN_n(\mathscr P)=\Theta(n).
\end{equation}
\item for every $A_0,B_0>0$, there is $c=c(A_0,B_0)>0$ such that, whenever
$1\le K\le c\log n$,
\begin{equation}
\label{eq:predicate-small-witness}
        \sup_{p\in[p_\eta,p_{1-\eta}]}
        \mu_p\bigl(\SmallN_{n,B_0K}(\mathscr P)\bigr)
        \le e^{-A_0K}.
\end{equation}
\end{enumerate}
Then
$$
        W_\eta(n)\le C_{\eta,\mathscr P}\frac n{\log n}.
$$
If conditions~(iii)--(iv) also hold with $\eta$ replaced by $\eta/4$, then the same bound
holds at central level $\eta$ in the uniform fixed-size model and in the model of independently
sampled constraints.
\end{thm}

\begin{proof}
Condition~(iv) implies the short-witness hypothesis of \cref{thm:abstract-window}, after choosing
$A_0,B_0$ to dominate the constants in the clean-booster lemma.  By
\eqref{eq:predicate-linear-scale}, throughout the central interval
\begin{equation}
\label{eq:predicate-MV-scale}
        p=\Theta(n^{1-k}),
        \qquad
        M_\eta(n)=\Theta(n),
        \qquad
        V_\eta(n)=\Theta(n).
\end{equation}
and in particular $p_{1-\eta}\le1/2$ for large $n$.  It remains to verify simulability.

Fix an arbitrary input constant $C_0>0$, let $1\le K\le c\log n$, and let
$A\subseteq\Omega_n(\mathscr P)$ be a clean booster with
$|A|\le C_0K$, and $ \mu_p((\F_n(\mathscr P))_{A\to1})
        \ge\mu_p(\F_n(\mathscr P))+e^{-C_0K}$.
Let $\beta=e^{-C_0K}$.  Let $R_0$ be the variables used by $A$.  Since $A$ is not a
cover, choose $\sigma_0\in\Sigma^{R_0}$ uncovered by $A$.  Add fresh variables and extend
$\sigma_0$ so that, on the resulting set $R$, every symbol occurs the same number
$r_*\ge L$ of times.  Then, 
$|R|=O_{\Sigma,k,C_0}(K)$.
After relabelling variables, assume $R$ is an initial segment of $[n]$.  This does not change
the law because $\mathscr P$ is closed under coordinate permutations.

Force the maximal family $M(R,\sigma)$ from
\eqref{eq:maximal-pattern-quasicover}.  It contains $A$, so the probability increment remains at
least $\beta$.  By \cref{lem:rigidity-exceptional-orbit}, the only uncovered fibres over $R$
are $F_{g\sigma}$, $g\in\Gamma$.  Put $m=n-|R|$.  Constraints disjoint from $R$ form an
independent random $\mathscr P$-instance on $m$ variables.  Let $B\subseteq\Sigma^m$ be
its uncovered set.  Because every $P\in\mathscr P$ is $\Gamma$-invariant, $B$ is
$\Gamma$-invariant.

Consider selected constraints meeting $R$ in exactly one vertex.  Since $R$ is an initial
segment, that vertex occupies the first coordinate of the local pattern.  On the fibre
$F_\sigma$, such a constraint induces a member of the indexed derivative family.  Conversely,
every indexed derivative has exactly $r_*$ extensions, one for each local vertex carrying its
specified symbol.  Hence the induced helpers form a $\theta$-random subfamily of
$$
        \binom{[m]}{k-1}\times\mathscr D(\mathscr P),
        \qquad
        \theta=1-(1-p)^{r_*},
$$
independently of $B$, and their expected number satisfies
\begin{equation}
\label{eq:predicate-helper-mean}
        \lambda
        \le r_*p\,|\mathscr D(\mathscr P)|\binom m{k-1}
        =O_{\eta,\mathscr P,C_0}(K).
\end{equation}
On the fibre $F_{g\sigma}$, the same crossing constraint induces the $g$-image of its helper
on $F_\sigma$.  Since $B=gB$, coverage of $B$ on the distinguished fibre automatically
covers all exceptional fibres.

The probability that any selected non-forced constraint meets $R$ in at least two vertices is
at most
\begin{equation}
\label{eq:predicate-multiple-intersections}
        Cp\sum_{j=2}^k |R|^j n^{k-j}
        =O_{\eta,\mathscr P,C_0}(K^2/n)
        =o(\beta).
\end{equation}
after decreasing the constant in $K\le c\log n$.  Let $E_{\ge2}$ denote this exceptional
event.  Off $E_{\ge2}$, the instance with $M(R,\sigma)$ forced is a cover if and only if
the induced derivative family covers $B$ on the distinguished fibre, and hence on every
exceptional fibre.  Therefore
$$
\begin{aligned}
        \Prob(B\text{ is covered by the binomial derivative family})
        &\ge \mu_p((\F_n(\mathscr P))_{M(R,\sigma)\to1})
              -\Prob(E_{\ge2})\\
        &\ge \mu_p(\F_n(\mathscr P))+\beta-o(\beta).
\end{aligned}
$$
A cover on the $m$ outside variables covers the original instance after the remaining variables
are added, so $\mu_p(\F_n(\mathscr P))\ge\mu_p(\F_m(\mathscr P))$.  Consequently, for large
$n$, the binomial derivative family covers $B$ with probability at least
\begin{equation}
\label{eq:predicate-binomial-help}
        \mu_p(\F_m(\mathscr P))+\frac\beta2.
\end{equation}
Since the indexed derivative family has $\Theta(m^{k-1})$ members,
\cref{cor:logarithmic-helper-extraction} gives
$d=O_{\eta,\mathscr P,C_0}(K)$ independent uniform derivative helpers with excess
$\alpha=e^{-O_{\eta,\mathscr P,C_0}(K)}$.

Apply \cref{prop:exceptional-fibre-simulability} with
$X=\Sigma^m$,
$\Omega'=\Omega_m(\mathscr P)$,
$\mathcal H=\F_m(\mathscr P)$, and with the indexed derivative family as $\mathcal Q$.  The pointwise comparison is precisely
\eqref{eq:predicate-derivative-replacement}.  Since $m=n-O(K)$ and
\eqref{eq:predicate-MV-scale} holds,
$$
        |\Omega'|p(1-p)=\Theta(n)\ge cV_\eta(n).
$$
Thus clean boosters are simulable.  The abstract theorem yields the Bernoulli bound.  If the
central-scale and witness hypotheses hold at level $\eta/4$, apply the Bernoulli conclusion at
that level.  Since $N_n(\mathscr P)=\Theta(n^k)=\Omega(n^2)$,
\cref{cor:linear-model-transfer} then gives the two fixed-size formulations at level $\eta$.
\end{proof}

\begin{rem}[Finite rigidity certificates]
\label[rem]{rem:finite-predicate-algorithm}
For fixed $(\Sigma,k,\mathscr P,\Gamma)$, the rigidity part of the theorem can be checked by
finite enumeration using \cref{prop:finite-rigidity-certificate}.  The remaining predicate-specific
question is the power comparison \eqref{eq:predicate-derivative-replacement}.  In the examples
below this reduces to comparing the number of coordinates frozen by a set of assignments.  Thus
the theorem separates the problem into a finite local classification and a quantitative extremal
inequality for the derivative family.
\end{rem}

\section{Pointwise cubical replacement over a finite alphabet}
\label{sec:cubical-replacement}

Fix integers $q\ge2$ and $k\ge2$.  A codimension-$j$ cylinder in $[q]^m$ is obtained by
fixing $j$ distinct coordinates to specified values.  Let $\U_{m,q}^{(j)}$ be the family of
codimension-$j$ cylinders, so
$|\U_{m,q}^{(j)}|=q^j\binom mj$.
For $B\subseteq[q]^m$, let $\rho_j(B)$ be the probability that a uniformly random
codimension-$j$ cylinder contains $B$.  If $B\neq\varnothing$, let $\ell(B)$ be the number
of coordinates on which all points of $B$ have the same value.  Then
$$
        \rho_j(B)=\frac{\binom{\ell(B)}j}{q^j\binom mj},
$$
with the convention that this is zero if $\ell(B)<j$.  If $B=\varnothing$, then
$\rho_j(B)=1$.

\begin{lem}[Pointwise cubical replacement]
\label[lem]{lem:q-pointwise}
Fix $q\ge2$ and $k\ge2$.  There are constants $c_{q,k},C_{q,k}>0$ such that if
$B\subseteq[q]^m$ and
$\rho_{k-1}(B)\ge\varepsilon\ge C_{q,k}m^{-(k-1)}$,
then
$\rho_k(B)\ge c_{q,k}\varepsilon^{k/(k-1)}$.
\end{lem}

\begin{proof}
The case $B=\varnothing$ is trivial.  Otherwise write $\ell=\ell(B)$.  If
$C_{q,k}$ is sufficiently large, then
$\rho_{k-1}(B)\ge C_{q,k}m^{-(k-1)}$ implies $\ell\ge2k$.  Hence
$$
        \rho_{k-1}(B)\asymp_{q,k}\left(\frac{\ell}{m}\right)^{k-1},
        \qquad
        \rho_k(B)\asymp_{q,k}\left(\frac{\ell}{m}\right)^k.
$$
The claimed inequality follows.
\end{proof}

\begin{cor}[Averaged cubical replacement with preserved excess]
\label{thm:q-cubical-replacement}
Fix $q\ge2$ and $k\ge2$.  Let $B$ be a random subset of $[q]^m$, independent of all
auxiliary cylinders below, and let $\zeta\in[0,1]$.  Suppose that $d\ge1$ independent
uniformly random codimension-$(k-1)$ cylinders cover $B$ with probability at least
$\zeta+\alpha$, where $\alpha>0$, and suppose that
$\frac{\alpha}{4d}\ge C_{q,k}m^{-(k-1)}$,
where $C_{q,k}$ is the constant from \cref{lem:q-pointwise}.  Then there is an integer
$$
        D\le C'_{q,k}
        d^{1+k/(k-1)}\alpha^{-k/(k-1)}\log(2d/\alpha)
$$
such that $D$ independent uniformly random codimension-$k$ cylinders cover $B$ with
probability at least $\zeta+\alpha/2$.
\end{cor}

\begin{proof}
Apply \cref{thm:averaged-power-replacement} with
$\mathcal Q=\U_{m,q}^{(k-1)}$, $\mathcal R=\U_{m,q}^{(k)}$,
$\gamma=k/(k-1)$, and the pointwise estimate from \cref{lem:q-pointwise}.
\end{proof}

\section{Random atomic CSPs}
\label{sec:atomic-csps}

Fix $q\ge2$ and $k\ge2$.  The variables take values in $[q]=\{0,\ldots,q-1\}$.  An
atomic constraint is specified by a $k$-set $S\subset[n]$ and a pattern $a\in[q]^S$; it
forbids assignments satisfying $x_S=a$.  In the dual covering language, this constraint is the
codimension-$k$ cylinder
$\{x\in[q]^n:x_S=a\}$.
Let $\C_{n,q,k}$ be the set of all atomic constraints, so
$$
        N_{n,q,k}=|\C_{n,q,k}|=q^k\binom nk.
$$
The random atomic CSP includes each coordinate of $\C_{n,q,k}$ independently with probability
$p$.  Let $\Cover_{n,q,k}\subseteq2^{\C_{n,q,k}}$ be the event that the chosen cylinders cover
all of $[q]^n$, equivalently that the CSP is unsatisfiable.

For $0<\eta<1/2$, define
$$
        W_{\eta,q,k}(n)=N_{n,q,k}(p_{1-\eta}-p_\eta),
$$
where $p_\alpha$ is the generalized inverse defined above for the event
$\Cover_{n,q,k}$.

The second delicate step is the passage from a minimally unsatisfiable atomic instance to
the $q$-ary Tarsi deficiency theorem.  We spell this out before using the deficiency bound in the
linear-location and small-witness estimates.

An atomic constraint on a variable set $S$ with forbidden pattern $a\in[q]^S$ is identified with
the codimension-$k$ subcube
$\{x\in[q]^R:x_S=a\}
$
inside the cube on any larger variable set $R\supseteq S$.  Thus an unsatisfiable atomic instance
on variables $R$ is the same thing as a subcube cover of $[q]^R$.

\begin{lem}[Minimal unsatisfiability gives a tight q-ary cover]
\label[lem]{lem:min-unsat-tight-cover}
Let $Y$ be an inclusion-minimal unsatisfiable atomic $q$-ary instance, and let $R$ be the set
of variables used by $Y$.  Then the subcubes corresponding to the constraints of $Y$, viewed
inside $[q]^R$, form a minimal cover of $[q]^R$ using all coordinates of $R$.  In particular,
$Y$ is an instance to which the q-ary Tarsi deficiency theorem applies.
\end{lem}

\begin{proof}
Every assignment $\sigma\in[q]^R$ extends to an assignment of all ambient variables.  Since $Y$
is unsatisfiable, this extension violates at least one constraint of $Y$.  That constraint depends
only on variables in $R$, so its corresponding subcube contains $\sigma$.  Hence the subcubes
cover $[q]^R$.

The cover is minimal.  Indeed, if a constraint $Q\in Y$ were redundant in the cover of $[q]^R$,
then $Y\setminus\{Q\}$ would still cover $[q]^R$, and therefore would still be unsatisfiable,
contradicting the inclusion-minimality of $Y$.  Finally, every coordinate of $R$ is used by some
constraint by the definition of $R$.  Thus the cover is tight in the sense needed for the Tarsi
bound.
\end{proof}

We use the following q-ary generalization of Tarsi's lemma for subcube covers.

\begin{thm}[q-ary Tarsi lemma for subcube covers]
\label{thm:q-tarsi}
Every tight minimal subcube cover of $[q]^v$ has size at least
$(q-1)v+1$.
\end{thm}

This is \cite[Theorem~3.11]{FilmusEtAl}, proved there as a special case of a matroid-cover
version of Tarsi's lemma.  There is also a direct and equivalent clause-set formulation in
\cite[Corollary~1.9.9]{Kullmann}.  Explicitly, an atomic constraint $Q=(S,a)$ corresponds to the
generalised clause
$$
        C_Q=\bigvee_{i\in S}(x_i\neq a_i).
$$
The assignments falsifying $C_Q$ are precisely the points of the subcube covered by $Q$.
Consequently, an atomic instance is unsatisfiable exactly when the associated generalised
clause-set is unsatisfiable, and inclusion-minimality agrees in the two languages.  For a uniform
alphabet of size $q$, Kullmann's weighted number of variables is
$$
        \mathrm{wn}(F)=\sum_{i\in\operatorname{var}(F)}(|D_i|-1)=(q-1)v,
$$
so his inequality $c(F)-\mathrm{wn}(F)\ge1$ gives the same bound.  Combining either formulation
with \cref{lem:min-unsat-tight-cover} gives the exact atomic statement used below.  For the Boolean
historical form, see Aharoni--Linial~\cite{AharoniLinial}, who explicitly attribute the result to an
unpublished observation of Tarsi.

\begin{cor}[Tarsi deficiency for minimal atomic CSPs]
\label[cor]{cor:atomic-tarsi}
Every inclusion-minimal unsatisfiable atomic $q$-ary instance using $v$ variables has at least
$(q-1)v+1
$
constraints.
\end{cor}

We next show that the coarse linear location of the atomic threshold is automatic in exactly the
range needed below.  Let
$a=q-1$, and $\delta=a(k-1)-1$. Thus $\delta>0$ is equivalent to
$$
        \gamma_{q,k}=(k-1)-\frac1{q-1}>0.
$$

\begin{prop}[Linear lower bound for atomic unsatisfiability]
\label[prop]{prop:linear-lower-atomic}
Fix $q\ge2$ and $k\ge2$, and suppose $(q-1)(k-1)>1$. There is $c=c(q,k)>0$ such that, if
$pN_{n,q,k}\le cn$, and $N_{n,q,k}=q^k\binom nk$, then
$\mu_p(\Cover_{n,q,k})=o(1).
$
Equivalently, a random atomic $q$-ary $k$-CSP with expected number of constraints at most
$cn$ is satisfiable with high probability.
\end{prop}

\begin{proof}
Write $N=N_{n,q,k}$.  Put $m_0=pN$, and assume $m_0\le cn$, where $c>0$ will be chosen
sufficiently small. For a set $S\subseteq[n]$, let $X_S$ be the number of selected atomic constraints whose support
is contained in $S$.  If $|S|=v$, then
$$
        X_S\sim\operatorname{Bin}\left(q^k\binom vk,p\right),
\quad \text{ and} \quad
        \E X_S
        =
        p q^k\binom vk
        =
        m_0\frac{\binom vk}{\binom nk}
        \le
        C_{q,k}c n\left(\frac vn\right)^k.
$$
Set $a=q-1$.  We claim that, with high probability, there is no set
$S\subseteq[n]$ with
$X_S\ge a|S|+1$.
It is enough to bound the event $X_S\ge av$.  By the standard binomial upper-tail estimate,
$$
        \Prob(X_S\ge av)
        \le
        \left(\frac{e\E X_S}{av}\right)^{av}
        \le
        \left(C_{q,k}c\left(\frac vn\right)^{k-1}\right)^{av}.
$$
Therefore the expected number of sets $S$ with $|S|=v$ and $X_S\ge av$ is at most
$$
\begin{aligned}
        \binom nv
        \left(C_{q,k}c\left(\frac vn\right)^{k-1}\right)^{av}
        &\le
        \left[
        \frac{en}{v}
        \left(C_{q,k}c\left(\frac vn\right)^{k-1}\right)^a
        \right]^v  \\
        &=
        \left[
        C'_{q,k}c^a
        \left(\frac vn\right)^{a(k-1)-1}
        \right]^v.
\end{aligned}
$$
Since $a(k-1)-1=\delta>0$, choose $c=c(q,k)>0$ so small that
$C'_{q,k}c^a<1/4$.  Then
$$
        \sum_{v=k}^n
        \left[
        C'_{q,k}c^a
        \left(\frac vn\right)^\delta
        \right]^v=o(1).
$$
Indeed, for fixed $L$, the contribution of $k\le v<L$ tends to zero, while the contribution of
$L\le v\le n$ is at most $\sum_{v\ge L}4^{-v}$, which can be made arbitrarily small by taking
$L$ large.

Now suppose the selected instance were unsatisfiable.  Choose an inclusion-minimal unsatisfiable
subinstance $Y$, and let $S$ be the set of variables used by $Y$.  Write $v=|S|$ and
$t=|Y|$.  By Corollary~\ref{cor:atomic-tarsi},
$t\ge(q-1)v+1=av+1$.
But all constraints of $Y$ are supported inside $S$, so $X_S\ge t\ge av+1$, contradicting the
high-probability event proved above.  Hence the whole instance is satisfiable with high probability.
\end{proof}

\begin{prop}[Linear upper bound for atomic unsatisfiability]
\label[prop]{prop:linear-upper-atomic}
Fix $q\ge2$ and $k\ge2$.  There is $C=C(q,k)>0$ such that, if
$pN_{n,q,k}\ge Cn$,
then
$\mu_p(\Cover_{n,q,k})=1-o(1)$.
\end{prop}

\begin{proof}
Let $Z$ be the number of satisfying assignments.  A fixed assignment $x\in[q]^n$ is ruled out
by exactly one atomic constraint on each $k$-set of variables.  Hence
$$
        \Prob(x\text{ satisfies all selected constraints})
        =
        (1-p)^{\binom nk}.
$$
Therefore
$$
        \E Z
        =
        q^n(1-p)^{\binom nk}
        \le
        \exp\left(n\log q-p\binom nk\right).
$$
Since
$$
        p\binom nk=\frac{pN_{n,q,k}}{q^k},
$$
we have $\E Z=o(1)$ as soon as $pN_{n,q,k}\ge Cn$ with $C>q^k\log q$.  Markov's inequality
then gives $Z=0$ with high probability.
\end{proof}

\begin{cor}[Automatic linear central scale]
\label[cor]{cor:automatic-linear-scale}
Fix $q\ge2$, $k\ge2$, and $0<\eta<1/2$.  Suppose
$\gamma_{q,k}=(k-1)-1/(q-1)>0$. Then there are constants $0<c_{\eta,q,k}<C_{\eta,q,k}<\infty$ such that every central point
$p\in[p_\eta,p_{1-\eta}]$ satisfies
$c_{\eta,q,k}n
        \le
        pN_{n,q,k}
        \le
        C_{\eta,q,k}n$.
Consequently,
$M_\eta(n)=\Theta_{\eta,q,k}(n)$, and $V_\eta(n)=\Theta_{\eta,q,k}(n)$.
\end{cor}

\begin{proof}
The lower bound on $pN_{n,q,k}$ follows from Proposition~\ref{prop:linear-lower-atomic}: if
$pN_{n,q,k}\le cn$ with $c$ sufficiently small, then
$\mu_p(\Cover_{n,q,k})=o(1)$, so $p$ cannot lie in the central interval for large $n$. The upper bound follows from Proposition~\ref{prop:linear-upper-atomic}: if $pN_{n,q,k}\ge Cn$ with $C$
sufficiently large, then $\mu_p(\Cover_{n,q,k})=1-o(1)$, so $p$ cannot lie in the central
interval. Since $pN_{n,q,k}=\Theta_{\eta,q,k}(n)$, and $p=O_{\eta,q,k}(n^{1-k})=o(1)$, we also have
$N_{n,q,k}p(1-p)=\Theta_{\eta,q,k}(n)$.
Thus $M_\eta(n)=\Theta_{\eta,q,k}(n)$ and $V_\eta(n)=\Theta_{\eta,q,k}(n)$.
\end{proof}

Let $\SmallN_s^{(q,k)}$ be the event that the random instance contains a subfamily of at most
$s$ atomic constraints covering all assignments on the variables it uses.  Equivalently, it contains
a minimal unsatisfiable atomic sub-CSP with at most $s$ constraints.

\begin{lem}[Small atomic covers are rare]
\label[lem]{lem:q-small}
Assume
$\gamma_{q,k}=(k-1)-1/(q-1)>0.
$
There are constants $c_{q,k},c'_{q,k}>0$ such that, if $s\le c_{q,k}\log n$, then uniformly
for all $p=O_{q,k}(n^{1-k})$,
$$
        \mu_p(\SmallN_s^{(q,k)})\le n^{-c'_{q,k}}.
$$
\end{lem}

\begin{proof}
It is enough to count inclusion-minimal unsatisfiable atomic subinstances with at most $s$
constraints.  Let such an instance have $t$ constraints and use $v$ variables.  By
Corollary~\ref{cor:atomic-tarsi},
$t\ge(q-1)v+1$, so $v\le \frac{t-1}{q-1}$. For fixed $t$ and $v$, the number of possible labelled subinstances using a prescribed
$v$-set of variables is at most
$\binom{q^k\binom vk}{t}$.
Therefore, by the union bound, the probability that there is such a subinstance with $t\le s$
constraints is at most
$$
        \sum_{1\le t\le s}
        \sum_{v\le (t-1)/(q-1)}
        \binom nv
        \binom{q^k\binom vk}{t}p^t .
$$
Using $p=O_{q,k}(n^{1-k})$, $v\le t$, and the standard estimates for binomial coefficients, the
summand is at most
$$
        \exp(C_{q,k}t\log t)\, n^{v-(k-1)t}
        \le
        \exp(C_{q,k}t\log t)\,
        n^{(t-1)/(q-1)-(k-1)t}.
$$
Since
$(k-1)-1/(q-1)=\gamma_{q,k}>0,
$
this is at most
$$
        \exp(C_{q,k}t\log t)\,n^{-\gamma_{q,k}t-1/(q-1)}.
$$
If $s\le c_{q,k}\log n$ and $c_{q,k}$ is sufficiently small, then
$\exp(C_{q,k}t\log t)\le n^{\gamma_{q,k}t/2}$ for all $t\le s$.  Summing over
$t\le s$ and the allowed values of $v$ gives $\mu_p(\SmallN_s^{(q,k)})\le n^{-c'_{q,k}}$, for
some constant $c'_{q,k}>0$.
\end{proof}

For the atomic model, take $\Sigma=[q]$, let
$\mathscr P_{q,k}^{\rm at}
        =\bigl\{\{a\}:a\in[q]^k\bigr\}$, and $\Gamma=\{\mathrm{id}\}$.

\begin{prop}[Atomic predicates are rigid and derivative-replaceable]
\label[prop]{prop:atomic-predicate-certificate}
The system $(\mathscr P_{q,k}^{\rm at},\Gamma)$ is rigid.  Its nonempty one-coordinate
derivatives are precisely the singleton patterns in $[q]^{k-1}$, with constant multiplicity,
and they satisfy \eqref{eq:predicate-derivative-replacement} with
$\gamma=\frac{k}{k-1}$, and $b=k-1$.
\end{prop}

\begin{proof}
For the singleton pattern system, $u\preceq_{\mathscr P}v$ holds exactly when $u=v$.
Thus \eqref{eq:local-rigidity-condition} forces $\tau=\sigma$, so rigidity holds, for example
with $L=k$.  Fixing the first coordinate of a singleton either gives the empty set or a
singleton in $[q]^{k-1}$; every nonempty derivative occurs with the same multiplicity.  Hence
random derivative constraints are exactly random codimension-$(k-1)$ cylinders, while native
constraints are codimension-$k$ cylinders.  The required power comparison is
\cref{lem:q-pointwise}.
\end{proof}

\begin{thm}[Atomic CSP window bound]
\label{thm:atomic-csp-window}
Fix $q\ge2$, $k\ge2$, and $0<\eta<1/2$.  Suppose
$
        \gamma_{q,k}=(k-1)-1/(q-1)>0.
$
Then
$$
        W_{\eta,q,k}(n)\le C_{\eta,q,k}\frac n{\log n}.
$$
\end{thm}

\begin{proof}
The linear central scale is \cref{cor:automatic-linear-scale}.  By \cref{lem:q-small}, for every
fixed $A_0,B_0>0$, the estimate
\eqref{eq:predicate-small-witness} holds after decreasing the constant in
$K\le c\log n$.  Proposition~\ref{prop:atomic-predicate-certificate} verifies rigidity and
power-replaceability.  Apply \cref{thm:finite-predicate-window}.
\end{proof}

\begin{cor}[Fixed-$k$ SAT]
\label[cor]{cor:ksat}
For every fixed $k\ge3$ and every $0<\eta<1/2$, the central clause window of random
$k$-SAT satisfies
$$
        W_{\eta,k}(n)\le C_{\eta,k}\frac n{\log n}.
$$
\end{cor}

\begin{proof}
Random $k$-SAT is the Boolean atomic model $q=2$.  The condition
$\gamma_{2,k}>0$ is exactly $k\ge3$.  Apply \cref{thm:atomic-csp-window}.
\end{proof}

\begin{cor}[Standard fixed-size atomic CSPs]
\label[cor]{cor:atomic-standard-models}
The conclusions of \cref{thm:atomic-csp-window,cor:ksat} hold with the same
$O(n/\log n)$ window in the uniform fixed-size model and in the standard model with
independently sampled constraints.
\end{cor}

\begin{proof}
This is already included in \cref{thm:finite-predicate-window}; alternatively apply
\cref{cor:linear-model-transfer} using
$N_{n,q,k}=q^k\binom nk=\Omega(n^2)$ and
\cref{cor:automatic-linear-scale}.
\end{proof}

\section{Random signed NAE-\texorpdfstring{$k$}{k}-SAT}
\label{sec:nae-sat}

We now give a second application of the abstract theorem.  The variables are Boolean.  A signed
NAE clause is specified by a $k$-set $S\subseteq[n]$ and an antipodal class
$[a]=\{a,\bar a\}\subseteq\bits^S$.  It is violated exactly by assignments whose restriction to
$S$ belongs to $[a]$.  Thus, in the covering language, a signed NAE clause is the union of two
antipodal codimension-$k$ cylinders,
$$
        \{x\in\bits^n:x_S=a\}\cup\{x\in\bits^n:x_S=\bar a\}.
$$
Let $\mathcal N_{n,k}$ be the set of all signed NAE clauses.  Then
$$
        N^{\NAE}_{n,k}=|\mathcal N_{n,k}|=2^{k-1}\binom nk.
$$
We include each element of $\mathcal N_{n,k}$ independently with probability $p$, and write
$\Cover^{\NAE}_{n,k}$ for the event that the selected violated sets cover $\bits^n$, equivalently
that the formula is NAE-unsatisfiable.  For $0<\eta<1/2$, define
$$
        W^{\NAE}_{\eta,k}(n)
        =N^{\NAE}_{n,k}(p_{1-\eta}-p_\eta),
        \qquad
        \mu_{p_\alpha}(\Cover^{\NAE}_{n,k})=\alpha.
$$

The analogue of the deficiency input is the following elementary observation.

\begin{lem}[Minimum degree in a minimal NAE-unsatisfiable formula]
\label[lem]{lem:nae-min-degree}
Let $Y$ be an inclusion-minimal NAE-unsatisfiable signed NAE-$k$-SAT instance.  If $Y$ has
$t$ clauses and uses $v$ variables, then every used variable has clause-degree at least two.
Consequently,
$v\le \frac{k}{2}t$.
\end{lem}

\begin{proof}
Suppose that a used variable $x$ occurs in exactly one clause $Q$.  By minimality,
$Y\setminus\{Q\}$ has a NAE-satisfying assignment.  Keep this assignment on all variables other
than $x$.  If the other $k-1$ signed literals in $Q$ already take both Boolean values, either
choice of $x$ satisfies $Q$.  Otherwise they all take the same value, and we choose $x$ so that
its signed literal takes the opposite value.  Since $x$ occurs nowhere else, this extends the
assignment to a satisfying assignment of $Y$, a contradiction.  Hence every used variable has
degree at least two, and counting variable--clause incidences gives $2v\le kt$.
\end{proof}

\begin{prop}[Linear lower bound for signed NAE-unsatisfiability]
\label[prop]{prop:nae-linear-lower}
Fix $k\ge3$.  There is $c=c(k)>0$ such that, if
$pN^{\NAE}_{n,k}\le cn$,
then
$\mu_p(\Cover^{\NAE}_{n,k})=o(1).
$
\end{prop}

\begin{proof}
Put $N=N^{\NAE}_{n,k}$ and $m_0=pN\le cn$.  For $S\subseteq[n]$, let $X_S$ be the number
of selected signed NAE clauses whose support is contained in $S$.  If $|S|=v$, then
$$
        X_S\sim\operatorname{Bin}\!\left(2^{k-1}\binom vk,p\right)
        \quad
\text{and}
\quad
        \E X_S
        =m_0\frac{\binom vk}{\binom nk}
        \le C_kcn\left(\frac vn\right)^k.
$$
Set $b=2/k$.  Choosing $c$ small enough that the relevant upper-tail threshold exceeds the
mean, the standard binomial estimate gives
$$
        \Prob(X_S\ge bv)
        \le
        \left(C_kc\left(\frac vn\right)^{k-1}\right)^{bv}.
$$
Therefore the expected number of $v$-sets $S$ with $X_S\ge bv$ is at most
$$
\begin{aligned}
        \binom nv
        \left(C_kc\left(\frac vn\right)^{k-1}\right)^{bv}
        &\le
        \left[
        C'_kc^b\left(\frac vn\right)^{b(k-1)-1}
        \right]^v.
\end{aligned}
$$
Here
$$
        b(k-1)-1=\frac{k-2}{k}>0.
$$
Choose $c=c(k)$ so that $C'_kc^b<1/4$.  For each fixed $L$, the contribution from
$k\le v<L$ tends to zero, while the contribution from $v\ge L$ is at most
$\sum_{v\ge L}4^{-v}$.  Thus, with high probability, no $S$ satisfies
$X_S\ge2|S|/k$.

If the selected formula were NAE-unsatisfiable, take an inclusion-minimal NAE-unsatisfiable
subformula $Y$, and let $S$ be its set of used variables.  Writing $v=|S|$ and $t=|Y|$,
\cref{lem:nae-min-degree} gives $t\ge2v/k$.  All clauses of $Y$ are supported inside $S$, so
$X_S\ge2v/k$, a contradiction.
\end{proof}

\begin{prop}[Linear upper bound for signed NAE-unsatisfiability]
\label[prop]{prop:nae-linear-upper}
Fix $k\ge2$.  There is $C=C(k)>0$ such that, if
$pN^{\NAE}_{n,k}\ge Cn$, then
$\mu_p(\Cover^{\NAE}_{n,k})=1-o(1).
$
\end{prop}

\begin{proof}
Let $Z$ be the number of NAE-satisfying assignments.  A fixed assignment violates exactly one
signed NAE clause on each $k$-set.  Hence
$$
        \E Z
        =2^n(1-p)^{\binom nk}
        \le \exp\left(n\log2-p\binom nk\right).
$$
Since $p\binom nk=pN^{\NAE}_{n,k}/2^{k-1}$, this tends to zero once
$pN^{\NAE}_{n,k}\ge Cn$ with $C>2^{k-1}\log2$.  Markov's inequality finishes the proof.
\end{proof}

\begin{cor}[Automatic linear central scale for signed NAE-SAT]
\label[cor]{cor:nae-linear-scale}
Fix $k\ge3$ and $0<\eta<1/2$.  Uniformly for
$p\in[p_\eta,p_{1-\eta}]$,
$pN^{\NAE}_{n,k}=\Theta_{\eta,k}(n)$, and $N^{\NAE}_{n,k}p(1-p)=\Theta_{\eta,k}(n)$.

\end{cor}

\begin{proof}
The lower and upper bounds follow from propositions \ref{prop:nae-linear-lower} and \ref{prop:nae-linear-upper}.  They also
give $p=O_{\eta,k}(n^{1-k})=o(1)$, so multiplication by $1-p$ does not change the order.
\end{proof}

Let $\SmallN^{\NAE}_{n,s}$ denote the event that the random formula contains an inclusion-minimal
NAE-unsatisfiable subformula with at most $s$ clauses.

\begin{lem}[Small signed NAE witnesses are rare]
\label[lem]{lem:nae-small}
Fix $k\ge3$.  There are constants $c_k,c'_k>0$ such that, if $s\le c_k\log n$, then uniformly
for all $p=O_k(n^{1-k})$,
$$
        \mu_p(\SmallN^{\NAE}_{n,s})\le n^{-c'_k}.
$$
\end{lem}

\begin{proof}
By \cref{lem:nae-min-degree}, a minimal NAE-unsatisfiable formula with $t$ clauses uses at most
$kt/2$ variables.  A union bound gives
$$
        \mu_p(\SmallN^{\NAE}_{n,s})
        \le
        \sum_{1\le t\le s}\sum_{v\le kt/2}
        \binom nv
        \binom{2^{k-1}\binom vk}{t}p^t.
$$
Using $p=O_k(n^{1-k})$, $v\le kt/2$, and the standard binomial-coefficient estimates, each
summand is at most
$$
        \exp(C_kt\log(t+1))\,n^{v-(k-1)t}
        \le
        \exp(C_kt\log(t+1))\,n^{-(k-2)t/2}.
$$
For $t\le c_k\log n$, the exponential factor is at most $n^{(k-2)t/4}$ for all sufficiently
large $n$.  Summing over $t$ and $v$ yields a fixed polynomial saving.
\end{proof}

Let
$$
        \mathscr P_k^{\NAE}
        =\bigl\{\{a,\bar a\}:a\in\bits^k\bigr\},
        \qquad
        \Gamma=\{\mathrm{id},\mathrm{comp}\}.
$$
Repeated antipodal pairs are identified in $\mathscr P_k^{\NAE}$.

\begin{lem}[Pointwise atomic-to-NAE replacement]
\label[lem]{lem:pointwise-atom-to-nae}
Fix $k\ge2$.  There are constants $c_k,C_k>0$ such that, if
$\rho^{\rm at}_{k-1}(B)\ge\varepsilon\ge C_km^{-(k-1)}$,
then
$\rho^{\NAE}_k(B)\ge c_k\varepsilon^{k/(k-1)}$,
where $\rho^{\rm at}_{k-1}$ is the probability that a random atomic
codimension-$(k-1)$ cylinder contains $B$, and $\rho^{\NAE}_k$ is the probability that
a random signed NAE clause contains $B$.
\end{lem}

\begin{proof}
The case $B=\varnothing$ is immediate.  Otherwise, let $\ell(B)$ be the number of coordinates
on which every point of $B$ takes the same value.  Then
$$
        \rho^{\rm at}_{k-1}(B)
        =\frac{\binom{\ell(B)}{k-1}}{2^{k-1}\binom m{k-1}}.
$$
For $C_k$ sufficiently large, the assumed lower bound implies $\ell(B)\ge2k$.  Whenever a
signed NAE clause chooses its $k$ coordinates among these frozen coordinates and takes the
antipodal class of the frozen pattern, it covers $B$.  Therefore
$$
        \rho^{\NAE}_k(B)
        \ge\frac{\binom{\ell(B)}k}{2^{k-1}\binom mk}.
$$
The two displays give the claimed power bound.
\end{proof}

\begin{prop}[Signed NAE predicates are rigid and derivative-replaceable]
\label[prop]{prop:nae-predicate-certificate}
For every $k\ge3$, the system
$(\mathscr P_k^{\NAE},\Gamma)$ is rigid.  Its one-coordinate derivatives, with their indexed
multiplicities, are uniformly distributed atomic singleton patterns in $\bits^{k-1}$, and they
satisfy \eqref{eq:predicate-derivative-replacement} with
$\gamma=\frac{k}{k-1}$, and 
$b=k-1$.
\end{prop}

\begin{proof}
For signed NAE patterns,
$u\preceq_{\mathscr P}v$ if and only if $u\in\{v,\bar v\}$.  Suppose
\eqref{eq:local-rigidity-condition} holds but $\tau\notin\{\sigma,\bar\sigma\}$.  Then there
are coordinates $i,j$ for which $\tau_i=\sigma_i$ and
$\tau_j\ne\sigma_j$, after globally complementing $\tau$ if necessary.  Extend
$\{i,j\}$ to a $k$-set.  On that set, $\tau$ is neither $\sigma$ nor its complement,
contradicting the local preorder condition.  Thus the system is rigid, for example with $L=k$.
Fixing one coordinate of an antipodal pair gives a singleton pattern; every singleton derivative
occurs with the same indexed multiplicity.  The power comparison is
\cref{lem:pointwise-atom-to-nae}.
\end{proof}

\begin{thm}[Signed NAE-$k$-SAT window bound]
\label{thm:nae-window}
For every fixed $k\ge3$ and every $0<\eta<1/2$,
$$
        W^{\NAE}_{\eta,k}(n)
        \le C_{\eta,k}\frac n{\log n}.
$$
\end{thm}

\begin{proof}
The central linear scale is \cref{cor:nae-linear-scale}.  Lemma~\ref{lem:nae-small} gives
\eqref{eq:predicate-small-witness} after decreasing the logarithmic range as a function of
$A_0,B_0$.  Proposition~\ref{prop:nae-predicate-certificate} supplies the finite predicate
certificate.  Apply \cref{thm:finite-predicate-window}.
\end{proof}

\begin{cor}[Standard fixed-size signed NAE-SAT]
\label[cor]{cor:nae-standard-models}
The conclusion of \cref{thm:nae-window} holds in the uniform fixed-size model and in the
standard model with independently sampled signed NAE clauses.
\end{cor}

\begin{proof}
This follows from the final assertion of \cref{thm:finite-predicate-window}.
\end{proof}

\section{Random hypergraph two-colourability}
\label{sec:hypergraph-two-colourability}

We now treat the unsigned model.  Let $\mathcal E_{n,k}=\binom{[n]}k$, and include every
$k$-set independently with probability $p$.  The resulting random $k$-uniform hypergraph is
written $H_k(n,p)$.  A two-colouring $x\in\bits^n$ is proper if no selected edge is
monochromatic.  In the covering language, an edge $S\in\mathcal E_{n,k}$ covers
$Q_S=\{x:x_S=0^S\}\cup\{x:x_S=1^S\}$.
Let $\mathsf{N2Col}_{n,k}\subseteq2^{\mathcal E_{n,k}}$ be the monotone event that the selected
hypergraph is not two-colourable.  Put $N^{\rm hyp}_{n,k}=\binom nk$, and define
$$
        W^{\rm hyp}_{\eta,k}(n)
        =N^{\rm hyp}_{n,k}(p_{1-\eta}-p_\eta),
        \qquad
        \mu_{p_\alpha}(\mathsf{N2Col}_{n,k})=\alpha.
$$

The elementary minimum-degree argument from \cref{lem:nae-min-degree} has the following unsigned
form.

\begin{lem}[Minimum degree in a minimal non-two-colourable hypergraph]
\label[lem]{lem:hyp-min-degree}
Let $Y$ be an inclusion-minimal non-two-colourable $k$-uniform hypergraph.  If $Y$ has
$t$ edges and uses $v$ vertices, then every used vertex has degree at least two.  Consequently,
$v\le \frac{k}{2}t$.
\end{lem}

\begin{proof}
Suppose that a used vertex $u$ lies in a unique edge $e$.  By minimality, $Y\setminus\{e\}$
has a proper two-colouring.  Keep this colouring away from $u$.  If $e\setminus\{u\}$ already
contains both colours, either colour for $u$ works; otherwise colour $u$ with the colour opposite
to the common colour on $e\setminus\{u\}$.  This properly colours $Y$, a contradiction.
Counting incidences now gives $2v\le kt$.
\end{proof}

\begin{prop}[Linear lower bound for non-two-colourability]
\label[prop]{prop:hyp-linear-lower}
Fix $k\ge3$.  There is $c=c(k)>0$ such that, if
$pN^{\rm hyp}_{n,k}\le cn$, 
then
$\mu_p(\mathsf{N2Col}_{n,k})=o(1)$.
\end{prop}

\begin{proof}
Put $N=N^{\rm hyp}_{n,k}$ and $m_0=pN\le cn$.  For $S\subseteq[n]$, let $X_S$ be the
number of selected edges contained in $S$.  If $|S|=v$, then
$$
        \E X_S=m_0\frac{\binom vk}{\binom nk}
        \le C_kcn\left(\frac vn\right)^k.
$$
Set $b=2/k$.  For $c$ sufficiently small, the standard binomial upper-tail estimate gives
$$
        \Prob(X_S\ge bv)
        \le
        \left(C_kc\left(\frac vn\right)^{k-1}\right)^{bv}.
$$
Hence the expected number of $v$-sets with $X_S\ge bv$ is at most
$$
        \left[
        C'_kc^b\left(\frac vn\right)^{b(k-1)-1}
        \right]^v.
$$
Since $b(k-1)-1=(k-2)/k>0$, choosing $c$ small and summing over $v\ge k$ shows that, with
high probability, no $S$ satisfies $X_S\ge2|S|/k$.  If the selected hypergraph were not
two-colourable, an inclusion-minimal non-two-colourable subhypergraph $Y$, on a vertex set $S$,
would satisfy $|Y|\ge2|S|/k$ by \cref{lem:hyp-min-degree}, a contradiction.
\end{proof}

\begin{prop}[Linear upper bound for non-two-colourability]
\label[prop]{prop:hyp-linear-upper}
Fix $k\ge3$.  There is $C=C(k)>0$ such that, if
$pN^{\rm hyp}_{n,k}\ge Cn$,
then
$\mu_p(\mathsf{N2Col}_{n,k})=1-o(1)$.
\end{prop}

\begin{proof}
Let $Z$ be the number of proper two-colourings.  A colouring with colour-class sizes $a$ and
$n-a$ has
$$
        \binom ak+\binom{n-a}k
$$
monochromatic potential edges.  By convexity, this quantity is minimized for a balanced colouring,
and for all sufficiently large $n$ it is at least $c_k\binom nk$, for a constant $c_k>0$.
Consequently,
$$
        \E Z
        \le 2^n(1-p)^{c_k\binom nk}
        \le \exp\left(n\log2-c_kpN^{\rm hyp}_{n,k}\right).
$$
This tends to zero once $pN^{\rm hyp}_{n,k}\ge Cn$ with $C>\log2/c_k$.  Markov's inequality
finishes the proof.
\end{proof}

\begin{cor}[Automatic linear central scale for hypergraph two-colourability]
\label[cor]{cor:hyp-linear-scale}
Fix $k\ge3$ and $0<\eta<1/2$.  Uniformly for $p\in[p_\eta,p_{1-\eta}]$,
$pN^{\rm hyp}_{n,k}=\Theta_{\eta,k}(n)$, and $N^{\rm hyp}_{n,k}p(1-p)=\Theta_{\eta,k}(n)$.
\end{cor}

\begin{proof}
Apply \cref{prop:hyp-linear-lower,prop:hyp-linear-upper}.  They also imply
$p=O_{\eta,k}(n^{1-k})=o(1)$.
\end{proof}

Let $\SmallN^{\rm hyp}_{n,s}$ be the event that the random hypergraph contains an
inclusion-minimal non-two-colourable subhypergraph with at most $s$ edges.

\begin{lem}[Small non-two-colourable witnesses are rare]
\label[lem]{lem:hyp-small}
Fix $k\ge3$.  There are constants $c_k,c'_k>0$ such that, if $s\le c_k\log n$, then
uniformly for $p=O_k(n^{1-k})$,
$$
        \mu_p(\SmallN^{\rm hyp}_{n,s})\le n^{-c'_k}.
$$
\end{lem}

\begin{proof}
By \cref{lem:hyp-min-degree}, a minimal witness with $t$ edges uses at most $kt/2$ vertices.
Therefore
$$
        \mu_p(\SmallN^{\rm hyp}_{n,s})
        \le
        \sum_{1\le t\le s}\sum_{v\le kt/2}
        \binom nv\binom{\binom vk}{t}p^t.
$$
Each summand is at most
$$
        \exp(C_kt\log(t+1))\,n^{v-(k-1)t}
        \le
        \exp(C_kt\log(t+1))\,n^{-(k-2)t/2}.
$$
For a sufficiently small constant in $s\le c_k\log n$, summing over $t,v$ gives a fixed
polynomial saving.
\end{proof}

Let
$$
        \mathscr P_k^{\rm hyp}=\bigl\{\{0^k,1^k\}\bigr\},
        \qquad
        \Gamma=\{\mathrm{id},\mathrm{comp}\}.
$$
For $B\subseteq\bits^m$, let $\rho^{\rm const}_{k-1}(B)$ be the probability that a random
constant codimension-$(k-1)$ cylinder contains $B$, and let
$\rho^{\rm hyp}_k(B)$ be the probability that the bad set of a random $k$-edge contains
$B$.

\begin{lem}[Pointwise constant-cylinder-to-edge replacement]
\label[lem]{lem:pointwise-constant-to-edge}
Fix $k\ge2$.  There are constants $c_k,C_k>0$ such that, if
$\rho^{\rm const}_{k-1}(B)\ge\varepsilon\ge C_km^{-(k-1)}$,
then
$\rho^{\rm hyp}_k(B)\ge c_k\varepsilon^{k/(k-1)}$.
\end{lem}

\begin{proof}
The case $B=\varnothing$ is immediate.  For $c\in\bits$, let
$L_c(B)=\{i\in[m]:x_i=c\text{ for every }x\in B\}$, and put
$\ell_c=|L_c(B)|$.  Then
$$
        \rho^{\rm const}_{k-1}(B)
        =\frac{\binom{\ell_0}{k-1}+\binom{\ell_1}{k-1}}
                    {2\binom m{k-1}}.
$$
Hence, for some $c\in\bits$,
$$
        \frac{\binom{\ell_c}{k-1}}{\binom m{k-1}}\ge\varepsilon.
$$
For $C_k$ sufficiently large, this implies $\ell_c\ge2k$.  Every $k$-set contained in
$L_c(B)$ has a bad set containing $B$, and therefore
$$
        \rho^{\rm hyp}_k(B)
        \ge \frac{\binom{\ell_c}k}{\binom mk}.
$$
The usual binomial-coefficient comparisons give the claim.
\end{proof}

\begin{prop}[The Property B predicate is rigid and derivative-replaceable]
\label[prop]{prop:hyp-predicate-certificate}
For every $k\ge3$, the system
$(\mathscr P_k^{\rm hyp},\Gamma)$ is rigid.  Its one-coordinate derivatives are the two constant
patterns $0^{k-1}$ and $1^{k-1}$, and they satisfy
\eqref{eq:predicate-derivative-replacement} with
$\gamma=\frac{k}{k-1}$, and $b=k-1$.
\end{prop}

\begin{proof}
For this pattern system, $u\preceq_{\mathscr P}v$ says that if $u$ is constant, then
$v$ is constant.  Let $\sigma$ have both colour classes of size at least $k$, and suppose
\eqref{eq:local-rigidity-condition} holds.  Then every $\tau$-monochromatic $k$-set is
$\sigma$-monochromatic.  If a colour class of $\tau$ has size at least $k$, it must lie
inside one colour class of $\sigma$; otherwise it contains a $k$-set meeting both
$\sigma$-classes.  Neither $\tau$-class can have size at most $k-1$, because the other
class would still meet both $\sigma$-classes and have size at least $k$.  Thus the two
$\tau$-classes are exactly the two $\sigma$-classes, and
$\tau\in\{\sigma,\bar\sigma\}$.  Hence rigidity holds with $L=k$.

Fixing the first coordinate of $\{0^k,1^k\}$ gives one of the two constant
$(k-1)$-patterns.  The derivative comparison is
\cref{lem:pointwise-constant-to-edge}.
\end{proof}

\begin{thm}[Random hypergraph two-colourability window]
\label{thm:hyp-window}
For every fixed $k\ge3$ and every $0<\eta<1/2$,
$$
        W^{\rm hyp}_{\eta,k}(n)
        \le C_{\eta,k}\frac n{\log n}.
$$
\end{thm}

\begin{proof}
The central linear scale is \cref{cor:hyp-linear-scale}.  Lemma~\ref{lem:hyp-small} gives the
arbitrarily strong logarithmic-scale witness estimate
\eqref{eq:predicate-small-witness}.  Proposition~\ref{prop:hyp-predicate-certificate} verifies
rigidity and derivative replacement.  Apply \cref{thm:finite-predicate-window}.
\end{proof}

\begin{cor}[Standard fixed-size hypergraph model]
\label[cor]{cor:hyp-standard-models}
The conclusion of \cref{thm:hyp-window} holds in the uniform $H_k(n,m)$ model and in the
model obtained by independently sampling $m$ edges.
\end{cor}

\begin{proof}
This follows from the final assertion of \cref{thm:finite-predicate-window}.
\end{proof}

\section{Concluding remarks}
\label{sec:concluding}

\subsection{The Abbe--Montanari reduction}

The scaling-window problem is connected to the still-open problem of proving convergence of the
critical density for random fixed-$k$ SAT.  Abbe and Montanari showed that a clause-window bound
of the form
$$
        W_{\eta,k}(n)=O\!\left(\frac{n}{\log^{1+\delta}n}\right)
$$
for some fixed $\delta>0$ would imply the satisfiability conjecture
\cite{AbbeMontanari}; see also \cite{PerkinsSurvey}.  At a high level, their reduction uses a window
narrow enough that the finite-size uncertainty in the critical density becomes summable along an
appropriate comparison of instance sizes.  One can then promote concentration around the
$n$-dependent critical points to convergence of those points.  Thus, a gain of any fixed positive
power beyond one logarithm would have consequences substantially stronger than a finite-size
scaling estimate.

Our result reaches the boundary scale $n/\log n$.  It improves the previous
$n/\log\log n$ upper bound by a factor of order $\log n/\log\log n$, but it does not meet the
strictly stronger hypothesis required by the Abbe--Montanari reduction.  It is therefore tempting to
seek an iteration of the present argument, or a more efficient replacement step, that extracts a better power of $\log n$.

The proof itself explains why a straightforward iteration cannot supply that extra power.  At scale
$K$, Theorem~\ref{thm:KLLM} gives a deterministic boost of size $\exp(-O(K))$.  The simulation
lemmas replace the clean booster by $D\le\exp(O(K))$ random coordinates, while the universal
layer estimate bounds the resulting increase by
$O\!\left({(D+1)}/{\sqrt{V_\eta(n)}}\right)$.
Consequently the strongest comparison available from these three ingredients has the form
$$
        \exp(-O(K))
        \le \frac{\exp(O(K))}{\sqrt{V_\eta(n)}},
$$
which forces only $K=\Omega(\log V_\eta(n))$.  Reapplying the same argument does not improve the
exponent: each new booster again has exponential strength and exponential simulation cost, and the
layer estimate depends only on the total number of sprinkled coordinates.

A superlogarithmic denominator would therefore require a genuinely stronger input at one of three
places: a Bourgain-type theorem producing a larger boost, a subexponential simulation of clean boosters,
or a model-specific replacement for the Boolean-layer estimate that exploits more geometry than
monotonicity alone.  

\subsection*{Acknowledgments}
The author is grateful to Imre Leader for his guidance and support. This work was supported
by the CB European PhD Studentship funded by Trinity College, Cambridge.

\end{document}